\documentclass[11pt,reqno]{amsart}
\usepackage{amsmath}
\usepackage{amsthm}
\usepackage{amssymb}
\usepackage{enumerate, paralist}
\usepackage[pagewise,mathlines]{lineno}

\usepackage{hyperref}

\usepackage{color}
\definecolor{darkgreen}{cmyk}{1,0,1,.2}
\definecolor{m}{rgb}{1,0.1,1}

\newdimen\theight
\def\TeXref#1{%
             \leavevmode\vadjust{\setbox0=\hbox{{\tt
                     \quad\quad  {\small \textrm #1}}}%
             \theight=\ht0
             \advance\theight by \lineskip
             \kern -\theight \vbox to
             \theight{\rightline{\rlap{\box0}}%
             \vss}%
             }}%

\usepackage{times}

\newcommand{\R}{\mathbf{R}}
\newcommand{\Q}{\mathbf{Q}}
\newcommand{\N}{\mathbf{N}}
\newcommand{\Z}{\mathbf{Z}}
\newcommand{\2}{\mathbf{2}}

\DeclareMathOperator{\diam}{diam}

\DeclareMathOperator{\Int}{Int}
\DeclareMathOperator{\op}{op}
\DeclareMathOperator{\Isom}{Isom}
\DeclareMathOperator{\supp}{supp}

\theoremstyle{plain}

\newtheorem{theorem}{Theorem}[section]
\newtheorem{lemma}[theorem]{Lemma}
\newtheorem{cor}[theorem]{Corollary}
\newtheorem{prop}[theorem]{Proposition}

\theoremstyle{definition}

\newtheorem{defn}[theorem]{Definition}
\newtheorem{example}[theorem]{Example}

\theoremstyle{remark}

\newtheorem{rem}{Remark}

\newtheorem{claim}{Claim}

\newtheorem{hyp}{Hypothesis}

\newcommand{\BB}{\mathcal{B}}
\newcommand{\CC}{\mathcal{C}}
\newcommand{\DD}{\mathcal{D}}
\newcommand{\EE}{\mathcal{E}}

\newcommand{\MM}{\mathcal{M}}

\newcommand{\OO}{\mathcal{O}}

\newcommand{\QQ}{\mathcal{Q}}

\newcommand{\VV}{\mathcal{V}}

\begin{document}

\title{On turbulent relations}

\author[J.A. \'Alvarez L\'opez]{Jes\'us A. \'Alvarez L\'opez}
\address{Departamento de Xeometr\'{\i}a e Topolox\'{\i}a\\
         Facultade de Matem\'aticas\\
         Universidade de Santiago de Compostela\\
         15782 Santiago de Compostela\\ Spain}
\email{jesus.alvarez@usc.es}

\author[A. Candel]{Alberto Candel} 
\address{Department of Mathematics\\ 
	California State University at Northridge\\ 
	18111 Nordhoff Street\\
	Northridge, CA 91330\\ U.S.A.}
\email{alberto.candel@csun.edu}

\thanks{Research of authors supported by Spanish Ministerio de Ciencia e Innovaci\'on grants MTM2011-25656 and MTM2008-02640}
\date{}

\maketitle

\begin{abstract}
  This paper extends the theory of turbulence of Hjorth to certain
  classes of equivalence relations that cannot be induced by Polish
  actions. It applies this theory to analyze the quasi-isometry
  relation and finite Gromov-Hausdorff distance relation in the space
  of isometry classes of pointed proper metric spaces, called the
  Gromov space.
\end{abstract}

\tableofcontents


\section{Introduction}\label{s:intro}

Gromov~\cite[Chapter~3]{Gromov1999},~\cite{Gromov1981} described a
space, which is called the Gromov space and denoted here by
\(\MM_*\), whose points are isometry classes of pointed, complete,
proper metric spaces, and which is endowed with a topology which
resembles the compact-open topology on the space of continuous
functions on \(\R\). The space \(\MM_*\) supports several equivalence
relations of geometric interest. For example, the relation of being
(coarsely) quasi-isometric, the relation of being at finite
Gromov-Hausdorff distance, the relation of being bi-Lipschitz
equivalent, and others.  Their dynamic complexity was reminiscent of
the complexity exhibited by the turbulent group actions of
Hjorth~\cite{Hjorth2000}, and this motivated the development of the
theory of turbulent relations carried out in this paper.

A section by section description of the contents of this paper now
follows.  In Section~\ref{s:cont rel} we analyze a topology on the
space of subsets of a space appropriate for working with equivalence
relations. This topology is essentially the Vietoris
topology~\cite{McAllister1978} but the properties that we need are not
found on the literature on the topic. These topological properties are
of a categorical nature, and are needed to obtain a new version
(Theorem~\ref{t:K-U}) of the Kuratowski-Ulam
theorem~\cite[p.~222]{Kuratowski1958} which describes how topological
properties of a subset of a space over which an equivalence relation
is defined translate to properties of the intersection of that set
with the orbits of the equivalence relation (indeed, our version of
the Kuratowski-Ulam theorem also applies to non-equivalence
relations). The Kuratowski-Ulam theorem is one of key tools for
studying generic ergodicity of one relation with respect to another.

In Section~\ref{s:classification} we briefly review the basic concepts
of classification of equivalence relations. Complexity of an
equivalence relation is quantified by comparing that relation with one
of the standard examples, like the identity relation over a space or
the relation ``being on the same orbit'' of a group action, for
instance.  Two concepts used for describing the relative complexity of
two equivalence relations, \(E\) over \(X\) and \(F\) over \(Y\), are
reducibility and generic ergodicity.  The relation \(E\) is Borel
reducible to \(F\), denoted by \(E\leq_BF\), if there is an
\((E,F)\)-invariant Borel mapping \(\theta : X\to Y\) (that is,
\(\theta\) takes equivalence classes of \(E\) into equivalence classes
of \(F\)) such that the mapping \(\bar\theta:X/E\to Y/F\) induced by
\(\theta\) between quotient spaces is injective. The relation \(E\) is
generically \(F\)-ergodic if for any \((E,F)\)-invariant Borel mapping
\(\theta: X\to Y\) there is a residual saturated subset \(C\subseteq X\)
such that the mapping \(\bar\theta:C/E\to Y/F\)  is constant. These
  notions were mainly studied for the orbit relation \(E_G^X\) (or
  simply \(E_G\)) of any action of a Polish group \(G\) on a Polish
  space \(X\) (a Polish action \(G\curvearrowright
  X\)).

The least complex equivalence relations, called smooth or concretely
classifiable, are those Borel reducible to the identity relation over a
standard Borel space.  For example, the equivalence relation of being
isometric in the set of compact metric spaces is smooth because the
space of equivalence classes of this relation is itself a Polish
metric space when endowed with the Gromov-Hausdorff metric.

At a higher level of complexity are the equivalence relations that are
classifiable by isomorphism classes of countable structures. A
countable structure is a structure on the natural numbers that is
determined by  countable many  relations. This set of countable
structures is endowed with a Polish topology, and carries a continuous
action of \(S_\infty\), the Polish group of permutations of the
natural numbers, so that two countable structures are isomorphic if
and only if they are in the same orbit of this \(S_\infty\)-action.
Thus, an equivalence relation over a Borel space is classifiable by
countable structures if it is Borel reducible to the relation given by
the action of \(S_\infty\) on the space of countable structures. A
variety of examples of equivalence relations that are classifiable by
countable structures and which arise in dynamical systems are given in
Kechris~\cite{Kechris2000}, Hjorth~\cite[Preface]{Hjorth2000}.

We can also consider the class of equivalence relations that are
  generically \(E_{S_\infty}^Y\)-ergodic for every Polish
  \(S_\infty\)-space \(Y\).  In particular, these equivalence
  relations are not classifiable by isomorphism classes of countable
  structures: roughly speaking, any attempt of classification of these
  relations by countable models becomes generically trivial.

A key concept in the analysis of the complexity of Polish group
actions (classification by countable structures and generic
ergodicity) is that of turbulence, introduced by
Hjorth~\cite{Hjorth2000}. For a Polish group action to be turbulent,
not only the action must be highly complex (transitive, minimal) but
the group itself must be highly complex (actions of locally compact
groups are not turbulent). Precisely, the action is turbulent when its
orbits are dense and meager, and its local orbits are somewhere dense,
where the local orbits are the orbits of any restriction of the given
action to a local action of an open identity neighborhood in the group
on an open subset of the space. If a Polish action
  \(G\curvearrowright X\) is turbulent, then \(E_G^X\) is generically
  \(E_{S_\infty}^Y\)-ergodic for any Polish \(S_\infty\)-space \(Y\)
  \cite[Theorem~3.18]{Hjorth2000}; in particular, \(E_G^X\) is not
  classifiable by isomorphism classes of countable
  structures. Moreover, assuming that \(E_G^X\) is Borel in \(X\times
  X\) for a Polish action \(G\curvearrowright X\), then \(E_G^X\) is
  not classifiable by isomorphism classes of countable structures if
  and only if \(X\) has a continuously \(G\)-embedded turbulent
  Polish \(G\)-space \cite{Hjorth2002}.

The relations of being at finite Gromov-Hausdorff distance and being
quasi-isometric in the Gromov space \(\MM_*\) are not reducible to an
equivalence relation given by a Polish group action
\cite{AlvarezCandel:Non-reduction}. In particular, these equivalence
relations are not classifiable by isomorphism classes of countable
structures. However it makes sense to study whether they are
generically \(E_{S_\infty}^Y\)-ergodic for any Polish
\(S_\infty\)-space \(Y\), which could be done by using some
appropriate version of turbulence. Therefore, the theory of turbulence
for group actions needs to be amplified to a theory of turbulence for
more general equivalence relations. This amplification is
carried out in this paper for a class of uniform equivalence
relations, which includes interesting examples like the above
  metric equivalence relations on the Gromov space.

A uniform equivalence relation is a pair, \((\VV,E)\), consisting of a
uniformity \(\VV\) with a distinguished entourage \(E\) which is an
equivalence relation.  A first example of uniform equivalence relation
arises from a Polish action \(G\curvearrowright X\). The
uniformity on \(X\) is generated by the entourages \(\{\, (x,gx)\mid
x\in X, g\in W\,\}\), where \(\{W\}\) is a neighborhood system of the
identity of \(G\), and the equivalence relation is \(E_G^X\). 
A second example arises from a distance-like mapping,
\(d:X\times X\to [0,\infty]\), that satisfies the standard properties
of a distance but it is allowed to have \(d(x,y)=\infty\) for some
\(x,y\in X\).  The uniformity is generated by the entourages
\(\{\, (x,x')\mid d(x,x')<\epsilon\,\}\), for \(\epsilon>0\), and the
equivalence relation \(E_d\) is given by \(x E_d y\) if and only if
\(d(x,y)<\infty\). The pair \((d, E_d)\) (or simply \(d\)) is called a
metric equivalence relation.

Generalizing the case of Polish actions, a uniform equivalence
relation \((\VV,E)\) on a space \(X\) is called turbulent when the
equivalence classes of \(E\) are dense and meager, and its local
equivalence classes are somewhere dense, where the local equivalence
classes are the equivalence classes of the equivalence relation on any
open subset \(U\subseteq X\) generated by \((U\times U)\cap V\) for any
entourage \(V\) of \(\VV\).

As said, the main goal of this paper is to develop the theory of
turbulence for a class of uniform equivalence relations and then
use it to analyze the complexity of several metric equivalence
relations in the Gromov space, which are not reducible to Polish
actions, proving that they are turbulent and, as a consequence,
generically \(E_{S_\infty}^Y\)-ergodic for any Polish
\(S_\infty\)-space \(Y\).  This analysis begins in
  Section~\ref{s:Turbulence and generic ergodicity}, where we
  introduce a class of metric equivalence relations, called of
  type~I. For any metric equivalence relations of type~I and any
  Polish \(S_\infty\)-space \(Y\), we show that turbulence implies
  generic \(E_{S_\infty}^Y\)-ergodicity. The results and proofs of
  Section~\ref{s:Turbulence and generic ergodicity} follow closely
  Hjorth's work, adapted to metric equivalence relations by using the
  concepts and preliminary results developed in the previous sections.
  The general theory is continued in Section~\ref{s:U_R,r}, where we
  give a sequence of hypothesis that collective-wise will eventually
  guarantee that a metric equivalence relation that satisfies them is
  of type~I and turbulent.

In Section~\ref{s:supremum}, as a prelude to the study of the
``turbulent dynamics'' of the Gromov space, we study the metric
equivalence relation \((d_\infty,E_\infty)\) on \(C(\R)\) defined by the
supremum distance, where \(C(\R)\) is equipped with the compact-open
topology.

Section~\ref{s:Gromov sp} reviews the construction of the Gromov space
\(\MM_*\), and the pointed Gromov-Hausdorff distance with possible
infinite values, \(d_{GH}\), between isometry classes of pointed
proper metric spaces. This distance defines the relation ``being at
finite Gromov-Hausdorff distance'' over \(\MM_*\), denoted by
\(E_{GH}\). Another equivalence relation over \(\MM_*\) introduced in
this section is ``being quasi-isometric,'' denoted by \(E_{QI}\),
which turns out to be induced by a distance function with possible
infinite values, \(d_{QI}\).

Sections~\ref{s:GH} and~\ref{s:QI} analyze the metric equivalence
relations given by \((d_{GH}, E_{GH})\) and \((d_{QI},E_{QI})\) over
\(\MM_*\).

Our analysis culminates in the following theorem.

\begin{theorem}\label{t: d_infty, d_GH and d_QI}
  If \((d,E)\) is \((d_\infty,E_\infty)\), \((d_{GH},E_{GH})\) or
  \((d_{QI},E_{QI})\), then:
\begin{enumerate}[\textup{(}i\textup{)}]

\item The metric equivalence relation \((d,E)\) is
turbulent.

\item \(E\) is generically \(E_{S_\infty}^Y\)-ergodic for every Polish
  \(S_\infty\)-space \(Y\).

\end{enumerate}
\end{theorem}

Parts~(ii) of this result applies to the case of \(Y\) being the
\(S_\infty\)-space of countable structures and thus can be seen as
justification of a metric space version of the so called Gromov's
principle for discrete groups: ``No statement about all finitely
presented groups is both non-trivial and true.''

\section{Continuous relations}\label{s:cont rel}

Let \( \2=\{ 0,1\}\) denote the two-point set. If \(X\) is any set,
then \(\2^X\), the set of mappings \(X\to \2\), is naturally identified
with the set of all subsets of \(X\) by means of the characteristic
mapping of a subset.

If \(A\subseteq X\), let
\[ 
  P_A = \{\, B\subseteq X \mid B\cap A\ne\emptyset\,\}.
\]
There is a natural identification
\begin{equation}\label{2 A=2 X setminus P X setminus A} \2^A = \2^X
  \setminus P_{X\setminus A}.
\end{equation}
Moreover \(P_\emptyset = \emptyset\) and \(P_X =
\2^X\setminus\{\emptyset\}\), and for any set \(I\subseteq \2^X\) of
subsets of \(X\), \(P_{\bigcup_{A\in I} A} = \bigcup_{A\in I} P_{A}\)
and \(P_{\bigcap_{A\in I}A} \subseteq\bigcap_{A\in I} P_{A}\). If
\(X\) is a topological space, then \(\2^X\) becomes a topological
space when endowed with the topology that has \(\{P_U\mid \text{\(U\)
  open in \(X\)}\}\) as a subbase. This is called the Vietoris
topology~(Vietoris~\cite{Vietoris1922}, Michael~\cite{Michael1951}).
In what follows, provided that \(X\) is a topological space and unless
otherwise stated, \(\2^X\) will always be endowed with the Vietoris
topology.

If \(\BB\) is a base for a topology on \(X\), then
\[
\biggl\{\,\bigcap_{U\in\CC}P_U \mid \text{\(\CC\) is a finite subset
  of \(\BB\)}\,\biggr\}
\] 
is a base for the Vietoris topology on \(\2^X\). It follows in particular that
\(\2^X\) is second countable if \(X\) is second countable.

A (binary) relation, \(E\), over sets, \(X\) and \(Y\), is a subset
\(E\subseteq X\times Y\). The sets \(X\) and \(Y\) are called the
\textit{source} and \textit{target} of \(E\), respectively. The
notation \(x E y\) means \((x,y)\in E\). For \(x\in X\), the (possibly
empty) set \(E(x)=\{\, y\in Y\mid x E y\,\}\) is called the
\textit{target fiber} of \(E\) over \(x\). The relation \(E\) is
completely specified by its target fiber map \(x\in X\mapsto
E(x)\in\2^Y\). More generally, the notation \(E(S)=\bigcup_{x\in
  S}E(x)\in\2^Y\) will be used for each \(S\subseteq X\). The target
fiber map can also be used to realize \(E(S)\) as a subset of
\(\2^Y\); the context will clarify this ambiguity.

\begin{defn}\label{defn:continuous_relation}
  A relation, \(E\), over two topological spaces, \(X\) and \(Y\), is
  called \textit{continuous} if the target fiber map \(x\in X\mapsto
  E(x) \in \2^Y\) is continuous.
\end{defn}

The following result follows directly from~\eqref{2 A=2 X setminus P X
  setminus A}.

\begin{lemma}[{\cite[Proposition
    2.1]{Michael1956-I}}]\label{l:cont. relation} 
  A relation \(E\subseteq X\times Y\) is continuous if and only if,  for every
  closed set \(F\subseteq Y\), the set \(
  \{x\in X\ \mid E(x)\subseteq F\} \) is closed in \(X\).
\end{lemma}

Let \(\pi_X\) and \(\pi_Y\) denote the factor projections of \(X\times
Y\) onto \(X\) and \(Y\), respectively. If \(A\subseteq X\),
\(B\subseteq Y\), and \(x\in X\), then
\begin{align}
  A\cap E^{-1}(P_B)&=\pi_X(E\cap(A\times B)), \label{A cap E -1(P B)=pi X(E cap(A times B)}\\
  E(x)&=\pi_Y(E\cap(\{x\}\times Y)). \label{E(x)=pi Y(E cap(x times
    Y)}
\end{align}
The following lemma is an easy consequence of~\eqref{A cap E -1(P
  B)=pi X(E cap(A times B)}.

\begin{lemma}\label{l:E cont. iff pi X:E to X open} A relation
  \(E\subseteq X\times Y\) is continuous if and only if the restriction
  \(\pi_X|_E : E\to X\) is an open mapping.
\end{lemma}

If \(E\) is a relation over \(X\) and \(Y\), then the
\textit{opposite} of \(E\) is the relation \(E^{\op}\) over \(Y\) and
\(X\) given by
\[
    E^{\op}=\{\,(y,x)\in Y\times X\ |\ x E y\,\}.
\]
The target fibers of \(E^{\op}\) are \(E^{\op}(y) = E^{-1}(P_{\{y\}})\),
and are called \textit{source fibers} of \(E\). Note that for all
\(A\subseteq X\) and all \(B\subseteq Y\),
\begin{align}
  (E^{\op})^{-1}(P_A) & = E(A), \label{(E op) -1(P A)}\\
  (E\cap(A\times B))^{\op} & = E^{\op} \cap (B\times A). \label{(E cap(A times B) op}
\end{align}
Because of~\eqref{(E op) -1(P A)}, \(E^{\op}:Y\to\2^X\) is continuous
if and only if, for any open set \(O\subseteq X\), the set \(E(O)\) is
open in \(Y\). In the case of equivalence relations, it is usually
said that \(E\) is open when this property is satisfied; this term is
now generalized to arbitrary relations.

\begin{defn}\label{defn:bi-continuous relation}
  A relation over topological spaces is called \textit{open} if its 
  opposite relation is continuous, and it is called
  \textit{bi-continuous} if it is both continuous and open.
\end{defn}

Relation \(E\) could also be open in the sense that the map
\(E:X\to\2^Y\) is open; this possible ambiguity will be clarified by
the context.

If \(E\) is a symmetric relation over a space \(X\), then the source
and target fibers are equal, and are simply called \textit{fibers} of
\(E\), and so \(E\) is bi-continuous if and only if \(E\) is
continuous.

\begin{example}\label{ex:cont. relation} The following are basic
  examples of continuous and bi-continuous relations.

  \begin{enumerate}[(i)]

  \item If \(E\) is the graph of a map \(f:X\to Y\), then \(E\)
    (respectively, \(E^{\op}\)) is continuous just when \(f\) is
    continuous (respectively, open). In particular, the diagonal
    \(\Delta_X\subseteq X\times X\) is a bi-continuous relation over
    \(X\) because it is the graph of the identity map of \(X\).

  \item If \(E\subseteq X\times Y\) is an open subset, then \(E\) is a
    bi-continuous relation over \(X\) and \(Y\).

  \item If \(E\) is a continuous relation over \(X\) and \(Y\), then
    \(E\cap(A\times V)\) is a continuous relation over \(A\) and \(V\), for any
    \(A\subseteq X\) and any open \(V\subseteq Y\). Thus, by~\eqref{(E cap(A
      times B) op}, if \(E\) is bi-continuous, then, for all open subsets
    \(U\subseteq X\) and \(V\subseteq Y\), the relation  \(E\cap(U\times V)\)
     over \(U\) and \(V\) is bi-continuous.

   \item An equivalence relation is bi-continuous precisely when the
     saturation of any open set is an open set. In particular, the
     equivalence relation defined by the orbits of a continuous group
     action is bi-continuous, and the equivalence relation defined by
     the leaves of a foliated space is also bi-continuous.

  \end{enumerate}
\end{example}

For any set of relations, \(R\subseteq 2^{X\times Y}\), and any
\(A\subseteq Y\), the following properties hold:
\begin{align}
  \Bigl(\bigcup_{E\in R}E\Bigr)^{-1}(P_A)&=\bigcup_{E\in R}E^{-1}(P_A),\label{(bigcup iE i) -1(P A)}\\
  \Bigl(\bigcap_{E\in R} E\Bigr)^{-1}(P_A)&\subseteq\bigcap_{E\in R}E^{-1}(P_A),\notag\\
  \Bigl(\bigcup_{E\in R}E\Bigr)^{\op}&=\bigcup_{E\in R}E^{\op},\label{(bigcup iE i) op}\\
  \Bigl(\bigcap_{E\in R}E\Bigr)^{\op}&=\bigcap_{E\in R}E^{\op}.\label{(bigcap
    iE i) op}
\end{align} 
The following result is a direct consequence of~\eqref{(bigcup iE i)
  -1(P A)} and~\eqref{(bigcup iE i) op}.

\begin{lemma}\label{l:union of cont and bi-cont rels} If \(R\) is a
  set of continuous \textup{(}respectively, bi-continuous\textup{)}
  relations over \(X\) and \(Y\), then \(\bigcup_{E\in R}E\) is a
    continuous \textup{(}respectively, bi-continuous\textup{)}
    relation over \(X\) and \(Y\).
\end{lemma}

\begin{rem}
  The intersection of two continuous relations is a relation that need
  not be continuous. For example, if \(E_1\) and \(E_2\) are the
  relations over \(\R\) given by the graphs of two different linear 
  mappings \(\R\to \R\), then 
  \(E_1\cap E_2=\{(0,0)\}\) is not a continuous relation.  However, the
  intersection of two continuous relations is continuous when one of
  the relations is also an open subset (Example~\ref{ex:cont.
    relation}-(ii)), as the next lemma shows.
\end{rem}

\begin{lemma}\label{l:intersection of a cont rel and an open set}
  Let \(E\) be a continuous \textup{(}respectively,
  bi-continuous\textup{)} relation over \(X\) and \(Y\), and let
  \(F\subseteq X\times Y\) be an open subset. Then \(E\cap F\) is
  continuous \textup{(}respectively, bi-continuous\textup{)} relation
  over \(X\) and \(Y\).
\end{lemma}

\begin{proof}
  Suppose that \(E\) is continuous. Let \(V\subseteq Y\) be an open
  set. For every \(x\in(E\cap F)^{-1}(P_V)\) there is  \(y\in(E\cap
  F)(x)\cap V=E(x)\cap F(x)\cap V\).Then \((x,y)\in F\) and, since \(F\) is an open subset of
  \(X\times Y\), there are open sets \(U\subseteq X\)
  and \(W\subseteq Y\) such that \((x,y)\in U\times W\subseteq F\). By
  Example~\ref{ex:cont. relation}-(iii), \(E\cap (U\times W) \) is a
  continuous relation over \(U\) and \(W\), and so \((E\cap(U\times
  W))^{-1}(P_V)\) is open in \(U\), hence in \(X\). Since \(
  x\in(E\cap(U\times W))^{-1}(P_V)\subseteq(E\cap F)^{-1}(P_V)\), this
  shows that \( (E\cap F)^{-1}(P_V)\) is open in \(X\), and hence that
  \(E\cap F\) is a continuous relation.
  
  If \(E\) is a bi-continuous relation, then \(E\cap F\) is a
  bi-continuous relation because of
  Example~\ref{ex:cont. relation}-(ii) and~\eqref{(bigcap iE i) op}.
\end{proof}

The \textit{composition} of two relations, \(E\subseteq X\times Y\) and
\(F\subseteq Y\times Z\), is the relation \(F\circ E\subseteq X\times Z\)
given by
\[
  F\circ E=\{ \, (x,z)\in X\times Z \mid \text{\(\exists y\in Y\) such that \(x E y\) and \(y F z\)}\, \}.
\]
Composition of relations is an associative operation and \(\Delta_X\)
is its identity at \(X\).  Moreover
\begin{equation}\label{(F circ E) op=E op circ F op} 
  (F\circ E)^{\op}=E^{\op}\circ F^{\op}.
\end{equation}
If \(E\subseteq X\times X\) is a relation, the symbol \(E^n\), for positive
\(n\in\N\), denotes the \(n\)-fold composition \(E\circ\dots\circ E\), and
\(E^0=\Delta_X\). If \(E'\subseteq X'\times
Y'\) is another relation over topological spaces, let \(E\times E'\) be
the relation over \(X\times X'\) and \(Y\times Y'\) given by
\[
  E\times E'=\{\, (x, x', y, y')\in X\times X'\times Y\times Y' \mid\text{\(x E y\) and \( x' E' y'\)} \,\}.
\]
Note that
\begin{equation}\label{(E times E') op=E op times E prime op}
  (E\times E')^{\op}=E^{\op}\times {E'}^{\op}.
\end{equation}

For relations \(E\subseteq X\times Y\) and \(G\subseteq X\times Z\), let
\((E,G)\) denote the relation over \(X\) and \(Y\times Z\) given by
\[
  (E,G)=\{\, (x,y,z)\in X\times Y\times Z\mid \text{\(x E y\) and \(x G z\)}\, \}.
\]

\begin{lemma}\label{l:F circ E is continuous}
  The following properties hold for relations:
  \begin{enumerate}[\textup{(}i\textup{)}]
  
  \item If \(E\) and \(F\) are continuous \textup{(}respectively,
    bi-continuous\textup{)} relations, then \(F\circ E\) is
    continuous \textup{(}respectively, bi-continuous\textup{)}
    relation.

  \item If \(E\) and \(E'\) are continuous \textup{(}respectively,
    bi-continuous\textup{)} relations, then \(E\times E'\) is a
    continuous \textup{(}respectively, bi-continuous\textup{)}
    relation.
    
  \item If \(E\) and \(G\) are continuous relations, then \((E,G)\) is a
    continuous relation.

  \end{enumerate}
\end{lemma}

\begin{proof}
  In~(i) and~(ii), the statements about continuity hold because
\begin{align*}
  (F\circ E)^{-1}(P_W) & = E^{-1}\left(P_{F^{-1}(P_W)}\right), \\
  (E\times E')^{-1}(P_{V\times V'}) & = E^{-1}(P_V)\times
  E^{\prime-1}(P_{V'}),
\end{align*}
for \(W\subseteq Z\), \(V\subseteq Y\) and \(V'\subseteq Y'\), and the statements
about bi-continuity follow from~\eqref{(F circ E) op=E op circ F op}
and~\eqref{(E times E') op=E op times E prime op}. Property~(iii) is a
consequence of~(i) and~(ii) since
\[
(F,G)=(F\times G)\circ(\Delta_X,\Delta_X),
\]
where \((\Delta_X,\Delta_X)\) is continuous because it is the graph of the
diagonal mapping \(x\mapsto(x,x)\).
\end{proof}

A consequence of Lemma~\ref{l:F circ E is continuous}-(i) is that the
continuous relations (and also the bi-continuous relations) over
topological spaces are the morphisms of a category with the operation
of composition. The assignment \(E\mapsto E^{\op}\) is a contravariant
functor of the category of bi-continuous relations to itself.

\begin{lemma}\label{l:E(x) is dense in Y}  The
  following properties hold for continuous relations over a
  topological space, \(X\), and a second countable topological space,
  \(Y\).
  \begin{enumerate}[\textup{(}i\textup{)}]
    
  \item If \(E\subseteq X\times Y\) is a continuous relation, then
    \[ 
    \{\, x\in X \mid \text{\(E(x)\) is dense in \(Y\)}\,\}
    \]
    is a \(G_\delta\) subset of \(X\).
    
  \item If \(E,F\subseteq X\times Y\) are continuous relations and
    \(E\subseteq F\), then
    \[
    \{\, x\in X \mid \text{\(E(x)\) is dense in \(F(x)\)}\,\}
    \]
    is a Borel subset of \(X\).
    
  \end{enumerate}
\end{lemma}

\begin{proof}
  Let \(\BB\) be a countable base of non-empty open sets for the
  topology of \(Y\).  Then Property~(i) is satisfies because
  \[
  \{\, x\in X\mid \text{\(E(x)\) is dense in
    \(Y\)}\,\}=\bigcap_{U\in\BB}E^{-1}(P_U),
  \]
the intersection of countably many open subsets of \(X\), 
  and Property~(ii) is satisfied because
  \begin{multline*}
    \{\, x\in X\mid \text{\(E(x)\) is dense in \(F(x)\)}\,\}\\
    \begin{aligned}
      &=\bigcap_{U\in\BB}\{\, x\in X\mid x\in F^{-1}(P_U)\Rightarrow x\in E^{-1}(P_U)\,\}\\
      &=\bigcap_{U\in\BB} \bigl( E^{-1}(P_U)\cup(X\setminus
      F^{-1}(P_U))\bigr),
    \end{aligned}
  \end{multline*}
the intersection of countably many Borel subsets of \(X\) (each the union of an
open set and a closed set).
\end{proof}

\begin{defn}\label{defn:transitive relation} An equivalence relation
  over a topological space is called \textit{topologically transitive}
  (respectively, \textit{topologically minimal}) if some equivalence
  class is dense (respectively, every equivalence class is dense).
\end{defn}

The following concepts and notation will be used frequently.

\begin{defn}
  \begin{enumerate}[(i)]

  \item A subset of a topological space is \textit{meager} if it is the
    countable intersection of nowhere dense subsets.

  \item A subset of a topological space is \textit{residual} if it contains the
    intersection of countably many open, dense subsets.

  \item A subset of a topological space has the \textit{Baire property} if it
    differs from an open set in a meager set.

  \item A topological space is \textit{Baire} if every residual subset is
    dense.

  \end{enumerate}
\end{defn}

\begin{defn}
Let \(P\) be a property that members of sets may or may not satisfy. Let
\(X\) be a topological space.

\begin{enumerate}[(i)]

\item Property \(P\) is satisfied by \textit{residually many} members of 
 \( X\), and denoted by \( (\forall^*x\in X)P(x)\), 
  if the set \( \{\,x\in X \mid P(x)\,\}\) is residual in \(X\).

\item Property \(P\) is satisfied by \textit{non-meagerly many}
  members of \(X\), and denoted by \( (\exists^* x\in X) P(x)\), if
  the set \(\{\, x\in X\mid P(x)\,\}\) is non-meager.

\end{enumerate}
\end{defn}

\begin{cor}\label{c:E(x) is dense in X} 
  If \(X\) is second countable and \(E\) is a topologically
  transitive, continuous equivalence relation over \(X\), then,
    \(\forall^*x\in X\), \( E(x) \) is dense in \(X\).
\end{cor}

\begin{proof}
  By Lemma~\ref{l:E(x) is dense in Y}-(i), the set
\[
       \{\, x\in X\mid \text{\(E(x)\) is dense in \(X\)}\, \}
\]
   is a dense \(G_\delta\) subset of \(X\).
\end{proof}

\begin{lemma}\label{l:E(x) cap A is residual in E(x)} 
  Let \(X\) be a metrizable topological space, let \(Y\) be a
  second countable topological space, and let \(E\subseteq X\times Y\)
  be a continuous relation. If every target fiber of \(E\) is a
  Baire space, then the following properties hold:
  \begin{enumerate}[\textup{(}i\textup{)}]
    
  \item If \(A\) is a \(G_\delta\) subset of \(Y\), then
    \[
    \{\, x\in X\mid \text{\(E(x)\cap A\) is residual in \(E(x)\)}\, \}
    \]
    is a \(G_\delta\) subset of \(X\).

  \item If \(B\) is an \(F_\sigma\) subset of \(Y\), then
    \[
    \{x\in X \mid \text{\(E(x)\cap B\) is non-meager in \(E(x)\)}\}
    \]
    is an \(F_\sigma\) subset of \(X\).
          
  \item If \(B\) is a Borel subset of \(Y\), then
    \[
    \{\, x\in X \mid \text{\(E(x)\cap B\) is residual in \(E(x)\)}\, \} 
    \]
    and
    \[ 
    \{\, x\in X \mid \text{\(E(x)\cap B\) is non-meager in \(E(x)\)}\,\}
    \]
    are Borel subsets of \(X\).
    
  \end{enumerate}
\end{lemma}

\begin{proof}
  To prove~(i), write \(A\) as an intersection \(A=\bigcap_{n\in \N} U_n\) of
  countable many open subsets \(U_n\subseteq Y\), and let \(\BB\) be a
    countable base for the topology of \(Y\). Then
  \begin{multline*}
    \{\, x\in X \mid \text{\(E(x)\cap A\) is residual in \(E(x)\)}\, \}\\
      \begin{aligned}
        &=\bigcap_{n\in \N}\{\, x\in X\mid \text{\(E(x)\cap U_n\) is residual in \(E(x)\)}\,\}\\
        &=\bigcap_{n\in \N}\{\, x\in X\mid \text{\(E(x)\cap U_n\) is dense in  \(E(x)\)}\,\}\\
        &=\bigcap_{n\in \N}\bigcap_{V\in\BB}\{\, x\in X\mid x\in E^{-1}(P_V)\Rightarrow x\in E^{-1}(P_{V \cap U_n})\,\}\\
        &=\bigcap_{n\in \N}\bigcap_{V\in\BB}(E^{-1}(P_{V \cap U_n})\cup(X\setminus E^{-1}(P_V))),
      \end{aligned}
    \end{multline*}
which is a \(G_\delta\) subset of \(X\); in fact, every \(E^{-1}(P_{V \cap
  U_n})\cup(X\setminus E^{-1}(P_V))\) is \(G_\delta\), because,  since
\(X\) is metrizable, closed
subsets of \(X\) are \(G_\delta\). 

   Property~(ii) is a consequence of~(i) because,  for every \(B\subseteq X\), 
    \begin{multline}\label{complement}
      \{\, x\in X\mid \text{\(E(x)\cap B\) is non-meager in \(E(x)\)} \,\}\\
      =X\setminus\{\, x\in X\mid \text{\(E(x)\cap(X\setminus B)\) is
        residual in \(E(x)\)}\,\}.
    \end{multline}
 
  To prove (iii), let \(\CC\) be the set of all Borel
  subsets \(B\subseteq Y\) such that, for any open subset \(U\subseteq Y\),
  the sets
\begin{gather}\label{E(x) cap U cap B is residual in E(x) cap U}
  \{\, x\in X\mid \text{\(E(x)\cap U\cap B\) is residual in \(E(x)\cap
    U\)}\,\}\\ 
    \intertext{and}
  \{\, x\in X\mid \text{\(E(x)\cap U\cap B\) is non-meager in \(E(x)\cap
    U\)}\,\}\label{E(x) cap U cap B is non-meager in E(x) cap U}
\end{gather}
are both Borel subsets of \(X\).

This \(\CC\) is a \(\sigma\)-algebra of subsets of \(Y\). Indeed, it
is closed under complementation, because of~\eqref{complement} and
Example~\ref{ex:cont. relation}-(iii), and it is also closed under
countable intersections, because if \(\{ C_n \mid n\in \N \}\) is a
countable set of members of \(\CC\), and \(U\subseteq Y\) is an open
set, then
\begin{multline*}
  \bigl\{\, x\in X\mid \text{\(E(x)\cap U\cap\bigcap_{n\in \N} C_n\) is residual in \(E(x)\cap U\)}\, \bigr\}\\
  =\bigcap_{n\in \N}\{\, x\in X\mid \text{\(E(x)\cap U\cap C_n\) is residual in \(E(x)\cap U\)}\,\}
\end{multline*}
is a Borel subset of \(X\), hence~\eqref{E(x) cap U cap B is residual in E(x) cap U}, and~\eqref{E(x) cap U cap B is non-meager in E(x) cap U}
follows from this: for any countable set \(\BB\) of open, non-empty,
subsets of \(U\) that is a base for the topology of \(U\), by
\cite[Proposition~8.26]{Kechris1995} (in a Baire space, a subset
  with the Baire property either is meager or is residual in some open
  set, but not both), and since the target fibers of
\(E\) are Baire spaces and \(\bigcap_{n\in \N} C_n\) has the Baire
property,
\begin{multline*}
  \bigl\{\, x\in X\mid \text{\(E(x)\cap U\cap\bigcap_{n\in \N} C_n\) is non-meager in \(E(x)\cap U\)}\,\bigr\}\\
  \begin{aligned}
  &=\bigcup_{V\in\BB}
  \bigl\{\, x\in X\mid \text{\(E(x)\cap V\cap\bigcap_{n\in \N} C_n\) is residual in \(E(x)\cap V\)}\,\bigr\}\\
  &=\bigcup_{V\in\BB}\bigcap_{n\in \N}\{\, x\in X\mid \text{\(E(x)\cap V\cap
    C_n\) is residual in \(E(x)\cap V\)}\,\}
    \end{aligned}
\end{multline*}
 is a Borel subset of \(X\), and so 
\(\bigcap_{n\in \N}C_n\in\CC\).

Every open subset  \(V\subseteq Y\) is a member of \(\CC\). Indeed, using
  Example~\ref{ex:cont. relation}-(iii), and applying~(i) and~(ii),
  \cite[Proposition~8.26]{Kechris1995}, the fact that \(E(x)\cap U\) is
  a Baire space, and the fact that open sets are \(F_\sigma\) because
  \(X\) is metrizable, it follows that
\[
    \{x\in X\ |\ E(x)\cap U\cap V\ \text{is non-meager in}\ E(x)\cap
    U\}=E^{-1}(P_{U\cap V}).
\] 
Consequently, \(\CC\) is the \(\sigma\)-algebra of all Borel subsets of
\(Y\), which establishes~(iii).
\end{proof}

\begin{lemma}\label{l:dense}
  Let \(E\subseteq X\times Y\) be an open relation over \(X\) and
  \(Y\). If \(A\subseteq B\subseteq Y\) and \(A\) is dense in \(B\), then
  \(E^{-1}(P_A)\) is dense in \(E^{-1}(P_B)\).
\end{lemma}

\begin{proof}
  Let \(O\) be an open subset of \(X\). Since \(E(O)\) is open in \(Y\) and
  \(A\) dense in \(B\),
    \begin{multline*}
      O\cap E^{-1}(P_B)\ne\emptyset\Longleftrightarrow E(O)\cap B\ne\emptyset\\
      \Longrightarrow E(O)\cap A\ne\emptyset\Longleftrightarrow O\cap
      E^{-1}(P_A)\ne\emptyset.\qed
    \end{multline*}
\renewcommand{\qed}{}
\end{proof}

\begin{lemma}\label{l:open and dense} 
  Let \(E\) be a bi-continuous relation over the topological spaces
  \(X\) and \(Y\), and assume that \(Y\) is second countable. If \(B\)
  is open and dense in \(Y\), then, \(\forall^*x\in X\), \(B\cap
  E(x)\) is open and dense in \(E(x)\).
\end{lemma}

\begin{proof}
  Let \(\{ V_n\mid {n\in \N}\}\) be a countable base for the topology
  of \(Y\).  Write
  \[
  O_n=\bigl(X\setminus E^{-1}(P_{V_n})\bigr) \cup E^{-1}(P_{V_n\cap
    B}).
  \]
  The boundary \(\partial E^{-1} (P_{V_n})\) is a meager set in \(X\)
  because \(E^{-1} (P_{V_n})\) is open in \(X\).  Since \(V_n\cap B\) is
  dense in \(V_n\), Lemma~\ref{l:dense} implies that \(E^{-1}(P_{V_n\cap
    B})\) is dense in \(E^{-1}(P_{V_n})\). Hence
  \[
    \bigl(X\setminus\overline{E^{-1}(P_{V_n})}\bigr) \cup E^{-1}(P_{V_n\cap B})
  \]
  is open and dense in \(X\setminus\partial E^{-1} (P_{V_n})\), and
  therefore the interior of \(O_n\) is open and dense in \(X\). This
  proves that \(\bigcap_{n\in \N} O_n\) is a residual subset of \(X\). If \(x\)
  is in \(\bigcap_{n\in \N}O_n\), then \(E(x)\cap B\) is dense in \(E(x)\),
  for otherwise there would be some \(V_n\) in the base  for which
  \(E(x)\cap B\cap V_n=\emptyset\) and \(E(x)\cap U_n\ne\emptyset\), which
  conflicts with the definition of \(O_n\).
\end{proof}

The following is a generalization of the Kuratowski-Ulam
Theorem~\cite[p.~222]{Kuratowski1958}.

\begin{theorem}[{Cf.\ {\cite[Theorem~8.41]{Kechris1995}}}]\label{t:K-U}
  Let \(E\) be a bi-continuous relation over the topological spaces
  \(X\) and \(Y\). Let \(Y\) be second
  countable, and let  \(A\subseteq Y\) have the Baire property. The
  following properties are satisfied:
  \begin{enumerate}[\textup{(}i\textup{)}]
  
  \item \(\forall^* x\in X\), \(A\cap E(x)\) has the Baire
    property in \(E(x)\);

  \item if \(A\) is meager in \(Y\), then, \(\forall^* x\in X\), \(A\cap
    E(x)\) is meager in \(E(x)\);

  \item if \(A\) is residual in \(Y\), then, \(\forall^* x\in X\),
    \(A\cap E(x)\) is residual in \(E(x)\).

  \end{enumerate}

 In addition, if \(X\) is a Baire space, \(E(X)\) is dense in
  \(Y\), and,  \(\forall^*x\in X\), \(E(x)\) is a
  Baire space,  the converses
  to~\textup{(}ii\textup{)} and~\textup{(}iii\textup{)} are also satisfied.
\end{theorem}

\begin{proof}
Lemma~\ref{l:open and dense} implies~(iii), which in turn implies~(ii). 

For~(i), if  \(A=U \triangle M\) for some meager set \(M\subseteq Y\) and some
open set \(U\subseteq Y\), then, for all \(x\in X\), \(A\cap E(x)\)
has the Baire property because
\[
A\cap E(x)=\bigl(U\cap E(x)\bigr) \triangle \bigl(M\cap E(x)\bigr),
\]
where \(U\cap E(x)\) is open in \(E(x)\), and, by~(ii),
\(\forall^*x\in X\), \(M\cap E(x)\) is meager in \(E(x)\).

Assume now that \(E(X)\) is dense in \(Y\) and that, \(\forall^*
  x\in X\), \(E(x)\) is a Baire space. Let \(A\) be a non-meager
subset of \(Y\) with the Baire property. Because of
\cite[Proposition~8.26]{Kechris1995}, there is a non-empty open
\(U\subseteq Y\) such that \(A\cap U\) is residual in \(U\); hence,
by~(iii), \(\forall^*x\in X\), \(A\cap U\cap E(x)\) is residual
in \(U\cap E(x)\). Because of
\cite[Proposition~8.22]{Kechris1995}, \(A\cap U\) has the Baire
property in \(X\), and thus in \(U\); hence, by~(i), 
  \(\forall^*x\in X\), \(A\cap U\cap E(x)\) has the Baire property in
\(U\cap E(x)\). Because \(E\) is continuous and \(E(X)\) is dense in
\(Y\), \(E^{-1}(P_U)\) is an open non-empty subset of
\(X\). Since, \(\forall^*x\in X\), \(E(x)\) is also a Baire
space, it follows from \cite[Proposition~8.26]{Kechris1995} that,
  \(\forall^*x\in E^{-1}(P_U)\), \(A\cap E(x)\) is not meager in
\(E(x)\).  Thus \(\exists^*x\in X\) such that \(A\cap E(x)\) is not
meager in \(E(x)\), by \cite[Proposition~8.26]{Kechris1995} since
  \(X\) is a Baire space. This proves the converse of~(ii), which in
turn implies the converse of~(iii).
\end{proof}

\begin{rem}\label{r:K-U}
  The classical Kuratovski-Ulam Theorem~(\textit{loc.~cit.},
  \textit{cf.}\ also \cite[Theorem~8.41]{Kechris1995}) is obtained
  from Theorem~\ref{t:K-U} in the case of Baire spaces by taking
  \(X=Y=X_1\times X_2\), where \(X_1\) and \(X_2\) are second countable
  Baire spaces, and \(E\) or \(E^{\op}\) equal to the
  equivalence relation whose equivalence classes are the fibers
  \(\{x_1\}\times X_2\) for \(x_1\in X_1\).
\end{rem}


\begin{cor}\label{c:K-U}
  Let \(X\) and \(Y\) be second countable Baire spaces, and let
  \( E \) be a bi-continuous relation over \(X\) and \(Y\). Suppose
  that: \(E\subseteq X\times Y\) is a Baire space, \(E(X)\) is dense in
  \(Y\), \(E^{\op}(Y)\) is dense in \(X\) and, \(\forall^*x\in X\),
  \(\forall^*y\in Y\), \(E(x)\) and \(E^{\op}(y)\) are Baire
  spaces. If \(F\subseteq X\times Y\) is such that \(F\cap E\) has the
  Baire property in \(E\), then, \(\forall^*y\in E(x)\), \(\forall^*x\in X\),
  \((x,y)\in F\) if and only if, \(\forall^*x\in E^{\op}(y)\),
    \(\forall^*y\in Y\), \((x,y)\in F\).
\end{cor}
  
\begin{proof} 
  Lemma~\ref{l:E cont. iff pi X:E to X open} implies that the
  restrictions of the projections \(\pi_X\) and \(\pi_Y\) to \(E\) are open
  mappings. Hence, by Example~\ref{ex:cont.  relation}-(i), their
  corresponding graphs, \(\Pi_{E,X} \subseteq E\times X\) and \(\Pi_{E,Y}
  \subseteq E\times Y\), are bi-continuous relations. Moreover, for \(x\in
  X\) and \(y\in Y\),
  \begin{gather*}
    \Pi_{E,X}^{\op}(x)=\{x\}\times E(x){\equiv E(x)},\\
    \Pi_{E,Y}^{\op}(y)= E^{\op}(y)\times\{y\}\equiv E^{\op}(y),\\
    A\cap\Pi_{E,X}^{\op}(x) =\{x\}\times(A\cap E)(x)\equiv(A\cap E)(x),\\
    A\cap\Pi_{E,Y}^{\op}(y) =(A\cap E)^{\op}(y)\times\{y\}\equiv(A\cap E)^{\op}(y),
  \end{gather*}
  and \(\Pi_{E,X}^{\op}(X)=\Pi_{E,Y}^{\op}(Y)=E\). Then, by Theorem~\ref{t:K-U} applied to \(E\), \(\Pi_{E,X}^{\op}\), \(\Pi_{E,Y}^{\op}\) and \(E^{\op}\),
  \begin{multline*}
    \forall^*y\in E(x),\ \forall^*x\in X,\ (x,y)\in F\\
    \begin{aligned}
      &\Longleftrightarrow\forall^*x\in X,\ (F\cap E)(x)\ \text{is residual in}\ E(x)\\
      &\Longleftrightarrow F\cap E\ \text{is residual in}\ E\\
      &\Longleftrightarrow\forall^*y\in Y,\ (F\cap E)^{\op}(y)\ \text{is residual in}\ E^{\op}(y)\\
      &\Longleftrightarrow\forall^*x\in E^{\op}(y),\
        \forall^*y\in Y,\ (x,y)\in F. \qedhere
    \end{aligned}
  \end{multline*}
\end{proof}

      
      
      

\begin{cor}\label{c:C} Let \(X\) and \(Y\) be second countable Baire
  spaces, and let \(E_n\subseteq X\times Y\) be countably many
  bi-continuous relations over \(X\) and \(Y\).
  The following properties hold:
  \begin{enumerate}[\textup{(}i\textup{)}]
      
  \item If \(A\subseteq X\) and \(B\subseteq Y\) are residual subsets,
    then there are residual subsets \(C\subseteq A\) and \(D\subseteq
    B\) such that, for all \(x\in C\), all \(y\in D\) and all
      \(n\in\N\), the set \(D\cap E_n(x)\) is residual in \(E_n(x)\)
    and \(C \cap E_n^{\op}(y)\) is residual in \(E_n^{op}(y)\).
      
  \item If \(X=Y\) and \(A\subseteq X\) is a residual subset, then there
    is a residual subset \(C\subseteq A\) such that, for all \(x\in C\)
    and all \(n\in\N\), \(C\cap E_n(x)\) is residual in \(E_n(x)\).
      
  \end{enumerate}
\end{cor}
  
\begin{proof}
  To prove~(i), define residual subsets, \(C_i\subseteq X\) and
  \(D_i\subseteq Y\), \(i\in \N\), by the following induction process on
  \(i\in\N\).  Set \(C_0=A\) and \(D_0=B\). Assuming that \(C_i\) and
  \(D_i\) have been defined, let
  \begin{align*}
    C_{i+1} &= \{\, x\in X\mid \text{\(\forall^*x\in X\), \(\forall n\in\N\), \(D_i\cap E_n(x)\) is residual in \(E_n(x)\)}\, \},\\
    D_{i+1} &= \{\, y\in Y\mid \text{\(\forall^*y\in Y\), \(\forall n\in\N\), \(C_i\cap E_n^{\op}(y)\) is residual in \(E_n^{\op}(y)\)}\,\}.
  \end{align*}
  By Theorem~\ref{t:K-U}, for all \( i\in \N\), \( C_i\) is residual
  in \(X\) and \(D_i\) is residual in \(Y\), and therefore \(
  C=\bigcap_{i\in \N}C_i \) is residual in \(A\) and \(D=\bigcap_{i\in
    \N}D_i\) is residual in \(B\) because \(A\) and \(B\) are
    dense in  \(X\) and \(Y\), respectively, since
    \(X\) and \(Y\) are Baire spaces.  Moreover, for all \(n\in \N\),
  all \( x\in C\) and all \(y\in D\), \(D\cap E_n(x)=\bigcap_{i\in \N}(D_i\cap
  E_n(x))\) is residual in \(E_n(x)\), and \(
  C\cap E_n^{\op}(y)=\bigcap_{i\in \N}(C_i\cap E_n^{\op}(y))\) is
  residual in \(E_n^{\op}(y)\).
  
To prove~(ii), let \(C_0=A\) and, assuming that \(C_i\) has been
defined, let
\[
    C_{i+1}=\{\, x\in X\mid \text{\(\forall^*x\in X\), \(\forall n\in\N\), \(C_i\cap E_n(x)\) is residual in \(E(x)\)}\, \}.
\]
By Theorem~\ref{t:K-U}, for all \( i\in \N\), \(C_i\) is residual in \(X\). 
Therefore \(C=\bigcap_{i\in \N}C_i\) is residual in \(A\)
  because \(A\) is dense in \(X\) since \(X\) is a Baire space, and, for all \(x\in C\) and all \(n\in\N\),
\( C\cap E_n(x)=\bigcap_{i\in \N} (C_i\cap E_n(x)) \) is residual in
\(E_n(x)\).
\end{proof}

\section{Classification and generic ergodicity}\label{s:classification}

Let \(X\) and \(Y\) be topological spaces, and let \(E\subseteq X\times X\)
and \(F\subseteq Y\times Y\) be equivalence relations. A mapping, \(\theta:X\to
Y\), is called \((E,F)\)-\textit{invariant} if
\[
  x E x' \Longrightarrow \theta(x) F \theta(x')
\]
for all \(x, x' \in X\). An \( (E,F)\)-invariant mapping \(\theta:X\to Y\)
induces a mapping, \(\bar\theta:X/E\to Y/F\), between quotient spaces.

The relation \(E\) is  \textit{Borel reducible} to \(F\),
written \(E\le_BF\), if there is an \((E,F)\)-invariant Borel mapping
\(\theta:X\to Y\) such that
\[
  x E x' \Longleftrightarrow \theta(x) F \theta(x')
\]
for all \(x, x' \in X\); \textit{i.e.}, such that the quotient mapping
\(\bar\theta:X/E\to Y/F\) is injective. If \(E\le_B F\) and \(F\le_BE\),
then \(E\) is said to be \textit{Borel bi-reducible} with \(F\), and is
denoted by \(E\sim_BF\).

The relation \(E\) is \textit{generically \(F\)-ergodic} if,
for any \((E,F)\)-invariant, Baire measurable mapping \(\theta:X\to Y\),
there is some residual saturated \(C\subseteq X\) such that
\(\bar\theta:C/(E\cap (C\times C))\to Y/F\) is constant.

\begin{rem}\label{r:generically F-ergodic}
  If \(E\) is a generically \(F\)-ergodic relation over \(X\), then every
  equivalence relation over \(X\) that contains \(E\) is also generically
  \(F\)-ergodic.
\end{rem}

The partial pre-order relation \(\le_B\) establishes a hierarchy on the
complexity of equivalence relations over topological spaces. Two key
ranks of this hierarchy are given by the following two concepts of
classification of relations. 

In the first one, \(E\) is said to be \textit{concretely classifiable}
(or \textit{smooth}, or \textit{tame}) if \(E\le_B\Delta_\R\) (the
identity relation on \(\R\)). This means that the equivalence classes
of \(E\) can be distinguished by some Borel mapping \(X\to\R\).

\begin{theorem}\label{t:Delta Y-ergodic} 
  Let \(X\) and \(Y\) be second countable topological spaces. If \(E\)
  is a continuous, topologically transitive equivalence relation over
  \(X\), then \(E\) is generically \(\Delta_Y\)-ergodic.
\end{theorem}

\begin{proof}
  Let \(\theta:X\to Y\) be \((E,\Delta_Y)\)-invariant and Baire
  measurable.  By \cite[Theorem~8.38]{Kechris1995}, \(\theta\) is
  continuous on some residual saturated set \(C_0\subseteq X\). By
  Corollary~\ref{c:E(x) is dense in X}, there is residual saturated
  \(C_1\subseteq X\) such that, for all \(x\in C_1\), 
  \(E(x)\) is dense in \(X\). Then \(C_0\cap C_1\) is a residual subset of \(X\) where \(\theta\)
  is constant.
\end{proof}

\begin{rem}
  In the above proof, if \(X\) is a Baire space, then \(C_0\cap
  C_1\neq\emptyset\).
\end{rem}

\begin{cor}[{Cf.\ {\cite[Theorem~3.2]{Hjorth2000}}}]\label{c:0-1 law} 
  Let \(X\) be a second countable space and let \(E\) be a continuous
  equivalence relation over \(X\). If \(E\) is topologically
  transitive, then  every \(E\)-saturated subset of \(X\) that has the
  Baire property is either residual or meager.
\end{cor}

\begin{proof}
  Apply Theorem~\ref{t:Delta Y-ergodic} to the characteristic function
  of the given subset.
\end{proof}

\begin{cor}\label{c:E is not concretely classifiable} 
  Let \(X\) be a second countable Baire space and let \(E\) be a
  continuous equivalence relation over \(X\). If \(E\) is
  topologically transitive and its equivalence classes are meager
  subsets of \(X\), then \(E\) is not concretely classifiable.
\end{cor}

\begin{proof}
  By Theorem~\ref{t:Delta Y-ergodic}, each \((E,\Delta_\R)\)-invariant
  Borel map \(\theta:X\to\R\) is constant on some residual saturated
  subset of \(X\). So \(\bar\theta:X/E\to\R/\Delta_\R\equiv\R\) cannot be
  injective because \(X\) is a Baire space and the equivalence
  classes are meager.
\end{proof}

The second classification concept can be defined by using
\(\prod_{n=1}^\infty\2^{\N^n}\) endowed with the product topology, which
is a Polish space. Each element of \(\prod_{n=1}^\infty\2^{\N^n}\) can be
considered as a structure on \(\N\) defined by a sequence \((R_n)\), where
each \(R_n\) is a relation over \(\N\) with arity \(n\). Two such structures
are isomorphic when they correspond by some permutation of \(\N\), which
defines the isomorphism relation \(\cong\) over
\(\prod_{n=1}^\infty\2^{\N^n}\). Then a relation \(E\) is
\textit{classifiable by countable structures} (or \textit{models}) if
\(E \le_B {\cong}\). This means that there is some Borel map
\(\theta:X\to\prod_{n=1}^\infty\2^{\N^n}\) such that \(x E x'\) if and only
if \(\theta(x) \cong \theta(x')\). Here, it is also possible to use the
structures on \(\N\) defined by arbitrary countable relational
languages, \textit{cf.}~\cite[Section~2.3]{Hjorth2000}.

The equivalence relation defined by the action of a group \(G\) on a set
\(X\) will be denoted by \(E_G^X\); in this case, the notation \(\OO(x)\)
will be used for the orbit of each \(x\in X\) instead of \(E_G^X(x)\). If
\(G\) is a Polish group, the set of all relations defined by
continuous actions of \(G\) on Polish spaces has a maximum with respect
to \(\le_B\), which is unique up to \(\sim_B\) and is denoted by
\(E_G^\infty\) \cite{BeckerKechris1996,KechrisMiller2004}.

As a special example, the group \(S_\infty\) of permutations of \(\N\)
becomes Polish with the topology induced by the product topology of
\(\N^\N\), where \(\N\) is considered with the discrete topology. Then the
canonical action of \(S_\infty\) on \(\prod_{n=1}^\infty\2^{\N^n}\)
defines the isomorphism relation \(\cong\) over the space of countable
structures, which is a representative of \(E_{S_\infty}^\infty\)
\cite{Hjorth2000}.

Classification by countable structures and generic ergodicity are well
understood for equivalence relations defined by Polish actions in
terms of a dynamical concept called \textit{turbulence} which was
introduced by Hjorth~\cite{Hjorth2000}.

\section{Turbulent uniform relations}\label{s:turbulent uniform}

A \textit{uniform equivalence relation}, or simply a \textit{uniform
  relation}, over a set, \(X\), is a pair, \((\VV, E)\), consisting of a
uniformity \(\VV\) on \(X\) and an equivalence relation \(E\) over \(X\) such
that \(E\in\VV\). Note that \((\VV,E)\) is determined by the entourages
(members of \(\VV)\) that are contained in \(E\), and that \(\VV\) induces a
uniform structure on each equivalence class of \(E\).

One important example of a uniform relation is that given by the
action of a topological group, \(G\), on a set, \(X\). This is of the form
\((\VV, E_G^X)\), where \(\VV\) is the uniform structure on \(X\) generated
by the entourages 
\begin{equation}\label{V_W} V_W=\{\,(x,gx) \mid x\in
  X,\ g\in W\,\},
\end{equation}
where \(W\) belongs to the neighborhood system of the identity of
\(G\). Thus a uniform relation over a topological space can be considered
as a generalized dynamical system.

Another important example of uniform relation is the following. A
\textit{metric} (or \textit{distance function}) \textit{with possible
  infinite values} on a set is a function \(d:X\times X\to
[0,\infty]\) satisfying the usual properties of a metric (\(d\) is
symmetric, equals \(0\) just on the diagonal of \(X\times X\) and
satisfies the triangle inequality). It defines an equivalence
relation, \(E_d^X\), over \(X\) given by \(x\,E_d^X\,y\) if and only
if \(d(x,y)<\infty\). There is a uniform relation induced by \(d\) of
the form \((\VV,E_d^X)\), where a base of \(\VV\) consists of the
entourages
\begin{equation}\label{V_epsilon}
    V_\epsilon = \{ \,(x,y)\in X\times X\mid d(x,y)<\epsilon\,\}.
\end{equation}
The term \textit{metric equivalence relation} (or \textit{metric
  relation}) will be used for the pair \((d,E_d^X)\) (or even for \(d\)).
Like the usual metrics, metrics with possible infinite values induce a
topology which has a base of open sets consisting of open balls;
unless otherwise indicated, the ball of center \(x\) and radius \(R\) will
be denoted by \(B_X(x,R)\) or \(B_d(x,R)\), or simply by \(B(x,R)\).

\begin{rem}
  Other generalizations of metrics also define uniform relations, like
  \textit{pseudo-metrics with possible infinite values}, defined in
  the obvious way, or when the triangle inequality is replaced by the
  condition \(d(x,y)\le \rho\bigl( d(x,z) + d(z,y)\bigr)\) for some
  \(\rho>0\) and all \(x,y,z\in X\) (\textit{generalized pseudo-metrics
    with possible infinite values}). They give rise to the concepts of
  \textit{pseudo-metric relation} and \textit{generalized
    pseudo-metric relation}.
\end{rem}

\begin{rem}\label{r:generalized pseudo-metric relations}
  Let \(d\) and \(d'\) be metric relations over \(X\) that
  induce respective uniform relations \((\VV,E)\) and \((\VV',E')\). If
  \(d'\le d\), then \(\VV\subseteq\VV'\) and \(E\subseteq E'\).
\end{rem}

\begin{defn}[{Cf.\ \cite[Definition~3.15]{Hjorth2000}}]\label{d:local equiv class} 
  Let \((\VV,E)\) be a uniform relation over a topological space
  \(X\). For any non-empty open \(U\subseteq X\) and any \(V\in\VV\) with
  \(V\subseteq E\), the set
    \[
      E(U,V)=\bigcup_{n=0}^\infty(V\cap(U\times U))^n
    \]
  is an equivalence relation over \(U\) called a \textit{local
   equivalence relation}. The \(E(U,V)\)-equivalence class of any \(x\in
  U\) is called a \textit{local equivalence class} of \(x\), and is denoted
  by \(E(x,U,V)\).
\end{defn}

For a relation given by the action of a group \(G\) on a space \(X\), the
local equivalence classes are called \textit{local orbits} in
Hjorth~\cite{Hjorth2000}, and the notation \(\OO(x,U,W)\) is used instead
of \(E_G^X(x,U,V)\) when \(V=V_W\) according to~\eqref{V_W}. Similarly,
for a uniform relation induced by a generalized pseudo-metric \(d\) on a
set \(X\), the notation \(E_d^X(x,U,\epsilon)\) is used instead of
\(E_d^X(x,U,V)\) when \(V=V_\epsilon\) according to~\eqref{V_epsilon}.

\begin{defn}[{Cf.~\cite[Definition~3.13]{Hjorth2000}}]\label{d:turbulent}
  A uniform relation is called \textit{turbulent} if:
    \begin{enumerate}[(i)]
  
      \item every equivalence class is dense,
  
      \item every equivalence class is meager, and
  
      \item every local equivalence class is somewhere dense.
  
      \end{enumerate}
    \end{defn}

\begin{rem}\label{r:turbulent}
  Definition~\ref{d:turbulent} does not correspond exactly to the
  definition of turbulence introduced by Hjorth for Polish actions
  \cite[Definition~3.13]{Hjorth2000}. To generalize exactly Hjorth's
  definition, condition~(iii) of Definition~\ref{d:turbulent} should
  be replaced with condition~(iii'):
    \begin{itemize}
    
    \item[(iii')] every equivalence class meets the closure of each
      local equivalence class.
    
    \end{itemize}
    In fact,~(i) already follows from~(iii'). In the case of Polish
    actions,~(iii) and~(iii') can be interchanged in the definition of
    turbulence by \cite[Lemmas~3.14 and~3.16]{Hjorth2000}; thus
    Definition~\ref{d:turbulent} generalizes Hjorth's definition. But
    in our setting, that equivalence is more delicate and our
    results become simpler by using~(iii).
\end{rem}

\begin{rem}\label{r:local equivalence equivalence classes}
  Let \((\VV,E)\) and \((\VV',E')\) be uniform relations over a topological
  space \(X\) such that \(\VV\subseteq\VV'\) and \(E\subseteq E'\). If the local
  equivalence classes of \((\VV,E)\) are somewhere dense
  (Definition~\ref{d:turbulent}-(iii)), then the local equivalence
  classes of \((\VV',E')\) are also somewhere dense.
\end{rem}

\begin{example}
  The following simple examples illustrate the generalization of the
  concept of turbulence for uniform relations.
    \begin{enumerate}[(i)]
  
    \item If \(E\) is an equivalence relation over a topological space
      \(X\), then  \(\VV=\{\,V\subseteq X\times X \mid E\subseteq V\,\}\)
      is a uniformity on \(X\), and \((\VV,E)\) is a uniform
      relation. Therefore \(E\) is the only entourage of \(\VV\) contained
      in \(E\), and \(E(x,U,E)=E(x)\cap U\) for any open \(U\subseteq X\) and
      all \(x\in U\), so it follows that \((\VV,E)\) is turbulent if the
      equivalence classes of \(E\) are dense and meager.
    
    \item Let \(G\) be a first countable topological group whose
      topology is induced by a right invariant metric \(d_G\). Suppose that \(G\) acts continuously on
      the left on a topological space \(X\). Then this action
      induces a pseudo-metric relation \(d\) on \(X\) with
      \(E_d^X=E_G^X\) and
     \[
          d(x,y)=\inf\{\,d_G(1_G,g)\mid g\in G,\ gx=y\,\}
      \]
      for \((x,y)\in E_G^X\), where \(1_G\) denotes the identity element
      of \(G\). The pseudo-metric relation \(d\) induces the same uniform relation
      as the action of \(G\) on \(X\), and therefore \(d\) is turbulent
      if and only the action is turbulent.
      
    \item Let \(\Z\) be the additive group of integers with the
      discrete topology, and let \(G\subseteq\Z^\N\) denote the topological
      subgroup consisting of the sequences \((x_n)\) such that \(x_n=0\)
      for all but finitely many \(n\in\N\). For some fixed irrational
      number \(\theta\), consider the continuous action of \(G\) on the
      circle \(S^1 \equiv \R/\Z\) given by
      \((x_n)\cdot[r]=[r+\theta\sum_nx_n]\), where \([r]\) is the element
      of \(S^1\) represented by \( r \in \R\). The orbits of this action
      are dense and countable. For each \(N\in\N\), the sets
      \[
      W_N=\{\,(x_n)\in G\mid\forall n\in\{0,\dots,N\},\ x_n=0\,\}
      \]
      are open and closed subgroups of \(G\) which form a base of neighborhoods
      of the identity element. The induced action of each \(W_N\) on
      \(S^1\) has the same orbits as \(G\); so
      \(\OO([r],U,W_N)=U\cap\OO([r])\) for all open \(U\subseteq S^1\)
      and each \([r]\in U\). It follows that this action is
      turbulent. In fact, the uniform equivalence relation induced by
      this action is of the type described in~(i): we have
      \(E_G^{S^1}\subseteq V\) for each entourage \(V\).  Moreover, for
      any invariant metric on \(G\), the induced pseudo-metric
      relation \(d\) on \(S^1\) is determined by \(d([r],[s])=\infty\)
      if \(\OO([r])\neq\OO([s])\) and \(d([r],[s])=0\)
      if \(\OO([r])=\OO([s])\). However, the action of \(G\) on
      \(S^1\) given by \((x_n)\cdot[r]=[r+\theta x_0]\) has the same
      orbits but is not turbulent: each point is a local orbit. Indeed
      this second action induces the same uniform equivalence relation
      as the action of \(\Z\) given by \(x\cdot[r]=[r+\theta x]\),
      which is not turbulent because \(\Z\) is locally compact.
  
    \end{enumerate}
\end{example}

\begin{defn}[{Cf.\ {\cite[Definition~3.20]{Hjorth2000}}}]\label{d:generically turbulent} 
  A uniform relation
  \((\VV,E)\) on a space \(X\) is \textit{generically turbulent} if:
    \begin{enumerate}[(i)]
  
    \item \(\forall^*x\in X\), the equivalence class of \(x\) is dense
      in \(X\),
  
    \item every equivalence class is meager,
      and
  
    \item \(\forall^*x\in X\), any local equivalence class of \(x\)
      is somewhere dense.
  
    \end{enumerate}
  \end{defn}
  
A metric relation is called (\textit{generically}) \textit{turbulent} if the induced uniform relation is (generically) turbulent.

\section{Turbulence and generic ergodicity}\label{s:Turbulence and generic ergodicity}

From now on, only metric relations over topological spaces will be
considered because that suffices for the applications given in this
paper. Some restriction on the topological structure of the space, and
some compatibility of that structure with the metric relation will be
required, and these are given in the following definition; they are
restrictive enough to prove the desired results, and general enough to
be satisfied in the applications.

 \begin{defn}\label{d:type I}
   A metric relation \(d\) on a space \(X\) is said to
   be of \textit{type~I} if:
  \begin{enumerate}[(i)]
      
  \item \(X\) is Polish;
      
  \item the topology induced by \(d\) on \(X\) is finer or equal than the
    topology of \(X\); and
      
  \item there is a set \(\EE=\{E_n\ |\ n\in\Z\}\) of relations over \(X\), with
    \(E_m\subseteq E_n\) if \(m\le n\), and such that:
    \begin{enumerate}[(a)]
    
    \item each \(E_n\in\EE\) is symmetric,
          
    \item each \(E_n\in\EE\) is a \(G_\delta\) subset of \(X\times X\),
    
    \item for each \(r>0\), there are some \(m\le n\) in \(\Z\) so that, for all \(x\in X\),
      \[
      E_m(x)\subseteq B_d(x,r)\subseteq E_n(x),
      \]
    
    \item for each \(n\in\Z\), there are \(r, s > 0\) such that, for
      all \(x\in X\),
      \[
      B_d(x,r)\subseteq E_n(x)\subseteq B_d(x,s),
      \]
          
    \item each \(E_n\in\EE\) is continuous, and
      
    \item for all \(k,m,n\in\Z\) and all \(x\in X\), if \(E_k\circ
        E_m\supseteq E_n\), then \(E_k\cap(E_m(x)\times E_n(x))\) is an
      open relation over \(E_m(x)\) and \(E_n(x)\).
  
    \end{enumerate}
  \end{enumerate}
\end{defn}

\begin{rem}\label{r:type I}
  In Definition~\ref{d:type I}, observe the following:

  \begin{enumerate}[(i)]
    
  \item Each \(E_n\in\EE\) is a \(G_\delta\) subset of \(X\)
    and, for each \(x\in X\), \(E_n(x)\equiv
    E_n\cap(\{x\}\times X)\) is a \(G_\delta\) subset of
    \(X\equiv X\times\{x\}\). Therefore, by
    \cite[Theorem~3.11]{Kechris1995}, \(E_n\) and
    \(E_n(x)\) are Polish subspaces of \(X\times X\) and \(X\),
    respectively; in particular, they are Baire spaces.
    
  \item Since \(E_d^X=\bigcup_{E\in\EE}E\), a metric relation of
    type~I is continuous by Lemma~\ref{l:union of cont and bi-cont
      rels}; however, its fibers need not be Polish spaces.
     
  \item By properties~(iii)-(a),(f), for all \(k,m,n\in\Z\) and all
    \(x\in X\), if \(E_k\circ E_n\supseteq E_m\), then
    \(E_k\cap(E_m(x)\times E_n(x))\) is a continuous relation over
    \(E_m(x)\) and \(E_n(x)\).
     
  \item It will become clear that the general results presented in
    this paper hold if the metric equivalence relation is of
    type~I only on some dense \(G_\delta\) subset. For the sake of
    simplicity, that generality is avoided since the
    conditions of Definition~\ref{d:type I} are satisfied in
    the applications to be given.
    
  \item Every \(E_k\) contains the diagonal \(\Delta_X\) by
      Definition~\ref{d:type I}-(iii)-(d). So \(E_k\circ E_l\supseteq
      E_l\), for all \(k,l\in\Z\).
        
  \end{enumerate}
\end{rem}

\begin{lemma}\label{l: type I, with EE'}
  Definition~\ref{d:type I}-(iii) holds if and only if there is a set of
  relations \(\EE'\) over \(X\) such that: 
    \begin{enumerate}[(a')]
    
    \item each \(E\in\EE'\) is symmetric,
          
    \item each \(E\in\EE'\) is a \(G_\delta\) subset of \(X\times X\),
    
    \item for each \(r>0\), there are some \(E,F\in\EE'\) so
      that, for all \(x\in X\), 
      \[
      E(x)\subseteq B_d(x,r)\subseteq F(x),
      \]
    
    \item for each \(E\in\EE'\), there are some \(r, s > 0\) so that
      \[
      B_d(x,r)\subseteq E(x)\subseteq B_d(x,s)
      \]
      for all \(x\in X\),
          
    \item each \(E\in\EE'\) is continuous, and
      
    \item for all \(E,F,G\in\EE'\) and \(x\in X\), if \(E\circ F\supseteq G\)
      then \(E\cap(F(x)\times G(x))\) is an open relation over \(F(x)\) and
      \(G(x)\).
  
    \end{enumerate}
\end{lemma}

\begin{proof}
  If \(\EE\) satisfies~(a)--(f), then it also satisfies~(a')--(f').

  Reciprocally, if~(a')--(f') are satisfied by \(\EE'\), then~(a)--(f)
    are satisfied by \(\EE=\{\,E_n\mid n\in\Z\,\}\), where each
    $E_n$ is chosen in $\EE'$ so that
  \begin{alignat*}{2}
  B_d(x,n)&\subseteq E_n(x)\subseteq B_d(x,n+1)&\quad&\text{for integers \(n\ge0\)},\\ 
  B_d(x,\textstyle{\frac{1}{-n+1}})&\subseteq E_n(x)\subseteq B_d(x,\textstyle{\frac{1}{-n}})&\quad&\text{for integers \(n<0\)}.\qed
  \end{alignat*}
\renewcommand{\qed}{}
\end{proof}

\begin{rem}\label{r: type I, with EE'}
  The variant of Definition~\ref{d:type I}, with the \(\EE'\)
    given by Lemma~\ref{l: type I, with EE'} replacing \(\EE\), will
    be useful in the applications.
\end{rem}

\begin{lemma}[{Cf.\ {\cite[Lemma~3.17]{Hjorth2000}}}]\label{l:theta}
  Let \(d\) be a metric relation of type~I over a space \(X\), and let
  \(\EE=\{\,E_n\mid n\in\Z\,\}\) be a sequence of subsets \(E_n\subseteq
  X\times X\)
  satisfying the conditions of Definition~\ref{d:type I}. Let \(G\) be a Polish
  group and let \(Y\) be a Polish \(G\)-space. If \(\theta:X\to Y\) is an
  \((E^X_d,E^Y_G)\)-invariant Borel map, then, for any neighborhood \(W\)
  of the identity element \(1_G\) in \(G\), \(\forall\ell\in\Z\),
  \(\forall^*x\in X\), and \(\forall^*x'\in E_\ell(x)\), there is some open
  neighborhood \(U\) of \(x\) in \(X\) such that, \(\forall k\in\Z\) and
  \(\forall^* x''\in U\cap E_k(x)\cap E_\ell(x')\), \(\exists g\in W\) so
  that \(g\cdot\theta(x)=\theta(x'')\). 
\end{lemma}

\begin{proof}
  Fix an open neighborhood \(W\) of \(1_G\) in \(G\).  The result follows
  from Corollary~\ref{c:K-U} and the following Claim~\ref{cl:theta}.
  
  \begin{claim}\label{cl:theta}
    \(\forall\ell\in\Z\), \(\forall x\in X\) and \(\forall^*x'\in
    E_\ell(x)\), there exists some open neighborhood \(U\) of \(x'\) in \(X\)
    such that, \(\forall k\in\Z\) and \(\forall^*x''\in U\cap E_k(x')\cap
    E_\ell(x)\), \(\exists g\in W\) so that
    \(g\cdot\theta(x')=\theta(x'')\).
  \end{claim}

  To prove this claim, let \(W'\) be a symmetric open neighborhood of the
  identity \(1_G\in G\) such that \({W'}^2\subseteq W\). Since \(G\) is a
  Polish group, there are countably many elements \( g_i\in G\),
  \(i\in\N\), such that \(G\subseteq \bigcup_{i\in \N} W'
  g_i\). Therefore, given \(\ell\in\Z\) and \(x\in X\), the set
  \(\theta(E_\ell(x))\subseteq \bigcup_{i\in \N}
  W'g_i\cdot\theta(x)\). The preimage of \(W'g_i\cdot\theta(x)\) via the
  mapping \( \theta : E_\ell(x)\to Y\) is analytic in \( E_\ell(x)\)
  because \(W'g_i\cdot \theta(x)\) is analytic
  \cite[Proposition~14.4-(ii)]{Kechris1995}. Hence it has the Baire
  property \cite[Theorem~21.6]{Kechris1995}, and so there are open subsets
  \(O_i\subseteq E_\ell(x)\) and residual subsets \(C_i\subseteq O_i\) such
  that \(\bigcup_iO_i\) is dense in \(E_\ell(x)\) and \(\theta(C_i)\subseteq
  W'g_i\cdot\theta(x)\). By using Definition~\ref{d:type I}-(iii)-(f)
  and Remark~\ref{r:type I}-(iii),(v) applied to the relation
  \(E_k\cap(E_\ell(x)\times E_\ell(x))\) over \(E_\ell(x)\), and by
  Corollary~\ref{c:C}-(ii) and Example~\ref{ex:cont. relation}-(iii),
  it follows that there is some residual \(D_i\subseteq C_i\) such that
  \(E_k(x')\cap D_i\) is residual in \(E_k(x')\cap O_i\) for all \(x'\in
  D_i\) and \(k\in\Z\).

  The union \( A= \bigcup_iD_i\) is residual in \(E_\ell(x)\).  If
  \(x'\in A\), then \(x'\in D_i\) for some \( i\) and so
  \(\theta(x')=g'g_i\cdot\theta(x)\) for some \(g'\in W'\). Let \(U\)
  be any open neighborhood of \(x'\) in \(X\) so that \(U\cap
  E_\ell(x)\subseteq O_i\). Then, \(\forall k\in\N\), \(U\cap
  E_k(x')\cap D_i\) is residual in \(U\cap E_k(x')\cap
  E_\ell(x)\). Moreover, for each \(x''\in E_k(x')\cap D_i\), there is
  \(g''\in W'\) so that
  \(\theta(x'')=g''g_i\cdot\theta(x)\). Therefore, if
  \[
    g=g''{g'}^{-1}\in W'{W'}^{-1}\subseteq W,
  \]
  then
  \[
    g\cdot\theta(x')=gg'g_i\cdot\theta(x)=g''g_i\cdot\theta(x)=\theta(x''),
  \]
  which completes the proof of Claim~\ref{cl:theta}.
  \end{proof}

\begin{cor}\label{c:theta}
  Under the conditions of Lemma~\ref{l:theta}, for every neighborhood
  \(W\) of the identity \(1_G\in G\) and, \(\forall^*x\in X\),
  \(\exists k\in\Z\) such that, \(\forall^*x'\in E_k(x)\), \(\exists
  g\in W\) for which \(g\cdot\theta(x)=\theta(x')\).
\end{cor}

\begin{proof}
  Let \(\ell\in\Z\) and let \(W\) be an open neighborhood of \(1_G\)
  in \(G\). Then, \(\forall^*x\in X\) and \(\forall^*x'\in
  E_\ell(x)\), let \(U\) be an open neighborhood of \(x\) in \(X\)
  satisfying the statement of Lemma~\ref{l:theta}. By
  Definition~\ref{d:type I}-(ii),(iii)-(c), there is \(k\le\ell\) so
  that \(E_k(x)\subseteq U\), obtaining that, \(\forall^* x''\in
  E_k(x)\cap E_\ell(x')\), \(\exists g\in W\) so that
  \(g\cdot\theta(x)=\theta(x'')\). Then the result follows from
  Theorem~\ref{t:K-U}, Definition~\ref{d:type I}-(iii)-(f) and
  Remark~\ref{r:type I}-(iii) with the relation
  \(E_\ell\cap(E_\ell(x')\times E_k(x'))\) over \(E_\ell(x')\) and
  \(E_k(x')\).
\end{proof}

\begin{theorem}[{Cf.\ {\cite[Theorem~3.18]{Hjorth2000}}}]\label{t:E_d^X is generically E_S infty Y-ergodic}
  Let \(d\) be a metric relation of type~I on a space \(X\) and let \(Y\)
  be a Polish \(S_\infty\)-space. If there are residually many \( x\in
  X\) for which every local equivalence class of \(x\) is somewhere dense,
  then \(E_d^X\) is generically \(E_{S_\infty}^Y\)-ergodic.
\end{theorem}

\begin{proof}
  Let \(\theta:X\to Y\) be an \((E_d^X,E_{S_\infty}^Y)\)-invariant
  Borel map. Let   \(\EE=\{\,E_n\mid n\in\Z\,\}\) be a sequence of
  subsets of \(E_x\subseteq X\times X\) satisfying the conditions of
  Definition~\ref{d:type I}. The sets
  \[
  W_N=\{\,h\in S_\infty\mid \forall \ell\le N,\ h(\ell)=\ell\,\},
  \]
  with \(N\in\N\), which are open and closed subgroups, form a base of neighborhoods of the identity
  \(1_{S_\infty}\in S_\infty\). Define
  \(I:X\times\N\to\N\cup\{\infty\}\) by setting \(I(x,N)\) equal to the
  least \(\ell\in\N\) such that, \(\forall^*x'\in E_{-\ell}(x)\), \(\exists
  h\in W_N\) so that \(h\cdot\theta(x)=\theta(x')\) if there is such
  \(\ell\), and setting \(I(x,N)=\infty\) if there is not such
  \(\ell\). Let \(\N\) and \(\N\cup\{\infty\}\) be endowed with the discrete
  topologies.
  
  \begin{claim}\label{cl:I is Baire measurable}
   The map \(I\) is Baire measurable.
  \end{claim}
  
  Since \( S_N=\{\,(y,h\cdot y)\mid y\in Y,\ h\in W_N\,\} \), \(N\in
  \N\), is analytic in \(Y\times Y\), and \(E_{-\ell}\), \(\ell\in
  \N\), is a Polish space by Remark~\ref{r:type I}-(i), the set
  \(R_{\ell,N}=E_{-\ell}\cap(\theta\times\theta)^{-1}(S_N)\) is
  analytic in \(E_{-\ell}\) \cite[Proposition~14.4-(ii)]{Kechris1995},
  and therefore it has the Baire property
  \cite[Theorem~21.6]{Kechris1995}. Hence there is an open set
  \(U_{\ell,N}\subseteq E_{-\ell}\) such that \(R_{\ell,N}\bigtriangleup
  U_{\ell,N}\) is meager in \(E_{-\ell}\). The restriction
  \(E_{-\ell}\to X\) of the first factor projection \(X\times X\to X\)
  is continuous and open by Lemma~\ref{l:E cont. iff pi X:E to X
    open}, so its graph \(\Pi_\ell\subseteq E_{-\ell}\times X\) is a
  bi-continuous relation (Example~\ref{ex:cont. relation}-(i)). By
  Theorem~\ref{t:K-U}-(ii), there is a residual set \(D_{\ell,N}\subseteq
  X\) such that, \(\forall x\in D_{\ell,N}\),
  \((R_{\ell,N}\bigtriangleup U_{\ell,N})\cap\Pi_\ell^{\op}(x)\) is
  meager in \(\Pi_\ell^{\op}(x)\). Note that
  \(\Pi_\ell^{\op}(x)=\{x\}\times E_{-\ell}(x)\equiv E_{-\ell}(x)\)
  and
  \begin{align*}
    (R_{\ell,N}\bigtriangleup U_{\ell,N})\cap\Pi_\ell^{\op}(x)
    &=\{x\}\times(R_{\ell,N}(x)\bigtriangleup U_{\ell,N}(x))\\
    &\equiv R_{\ell,N}(x)\bigtriangleup U_{\ell,N}(x).
  \end{align*}
  Hence, \(\forall x\in D_{\ell,N}\), \(R_{\ell,N}(x)\bigtriangleup U_{\ell,N}(x)\) is meager in
  \(E_{-\ell}(x)\). On the other hand,
  \[
    I^{-1}(\{0,\dots,\ell\})=\bigcup_{N=0}^\infty(Q_{\ell,N}\times\{N\}),
  \]
  where
  \[
    Q_{\ell,N}=\{\,x\in X\mid(E_{-\ell}\cap R_N)(x)\ \text{is residual in}\ E_{-\ell}(x)\,\}.
  \]
  Since
  \[
  Q_{\ell,N}\cap D_{\ell,N}=\{\,x\in D_{\ell,N}\mid(E_{-\ell}\cap
  U_N)(x)\ \text{is dense in}\ E_{-\ell}(x)\,\},
  \]
  it follows that \(Q_{\ell,N}\) has the Baire property in \(X\) by
  Lemmas~\ref{l:intersection of a cont rel and an open set}
  and~\ref{l:E(x) is dense in Y}-(ii), and the proof of
  Claim~\ref{cl:I is Baire measurable} is finished.
  
  By \cite[Theorem~8.38]{Kechris1995}, Claim~\ref{cl:I is Baire
    measurable}, and Corollary~\ref{c:theta}, there is a dense
  \(G_\delta\) subset \(C_0\subseteq X\) such that \(\theta\) is continuous on
  \(C_0\), \(I\) is continuous on \(C_0\times\N\), and
  \(I(C_0\times\N)\subseteq\N\).
  
  For \(k\in\Z\), a non-empty open set \(U\subseteq X\) and \(x\in U\),
  define
  \[
  \QQ(x,U,k)=\bigcup_{i=0}^\infty(E_k\cap(U\times U))^i(x).
  \]
  The following properties are consequences of Definition~\ref{d:type
    I}-(iii)-(c),(d):
  \begin{itemize}
      
  \item for each \(\epsilon>0\), there is \(k\in\Z\) such that
    \(\QQ(x,U,k)\subseteq E_d^X(x,U,\epsilon)\) for all \(x\in U\), and
      
  \item for each \(k\in\Z\), there is \(\epsilon>0\) such that
    \(E_d^X(x,U,\epsilon)\subseteq\QQ(x,U,k)\) for all \(x\in U\).
      
  \end{itemize}
  Hence, by hypothesis, there is a residual set \(C_1\subseteq X\) such
  that, for any \(U\), \(x\) and \(k\) as above, if \(x\in C_1\), then
  \(\QQ(x,U,k)\) is somewhere dense. By Corollary~\ref{c:C}-(ii), there
  is a residual set \(C\subseteq C_0\cap C_1\) such that, for all \(x\in C\) and all \(k\in\Z\), \(E_k(x)\cap C\) is
  residual in \(E_k(x)\).
  
  Fix \(x\), \(y\) in \(C\) and a complete metric inducing the
  topology of \(Y\).
  
  \begin{claim}\label{cl:E d X is generically E S infty Y-ergodic}
    There exist sequences, \((x_i)\) and \((y_i)\) in \(C\) with \(x_1=x\) and
    \(y_1=y\), \((g_i)\) and \((h_i)\) in \(S_\infty\), \((U_i)\) and \((V_i)\)
    consisting of open subsets of \(X\), and \((n_i)\) and \((k_i)\) in
    \(\N\), such that:
    \begin{enumerate}[(i)]
      
    \item \(g_i\cdot\theta(x)=\theta(x_i)\);
        
    \item \(h_i\cdot\theta(y)=\theta(y_i)\);
        
    \item \(x_{i+1}\in U_{i+1}\cap C\cap\QQ(x_i,U_i,-n_i)\);
        
    \item \(y_{i+1}\in V_{i+1}\cap C\cap\QQ(y_i,V_i,-k_i)\);
        
    \item \(U_i\supseteq V_i\supseteq U_{i+1}\);
        
    \item \(\diam(\theta(U_i\cap C))<2^{-i}\);
        
    \item \((U_{i+1}\cap C)\times\{N_{i+1}\}\subseteq I^{-1}(n_{i+1})\)
      for
      \[
        N_{i+1}=\max\{\,g_{i+1}(\ell),g_{i+1}^{-1}(\ell)\mid\ell\le i+1\,\};
      \]

    \item \((V_{i+1}\cap C)\times\{K_i\}\subseteq I^{-1}(k_i)\) for
      \[
      K_i=\max\{\,h_i(\ell),h_i^{-1}(\ell)\mid \ell\le i\,\};
      \]
          
    \item \(g_{j+1}(\ell)=g_{i+1}(\ell)\) and
      \(g^{-1}_{j+1}(\ell)=g^{-1}_{i+1}(\ell)\) for \(\ell\le i+1\le
      j+1\);
        
    \item \(h_j(\ell)=h_i(\ell)\) and \(h^{-1}_j(\ell)=h^{-1}_i(\ell)\)
      for \(\ell\le i\le j\);
        
    \item \(\QQ(x_i,U_i,-n_i)\cap V_i\) is dense in \(V_i\); and
        
    \item \(\QQ(y_i,V_i,-k_i)\cap U_{i+1}\) is dense in \(U_{i+1}\).
        
    \end{enumerate}
  \end{claim}
  
  If this assertion is true, then there exist \(g=\lim_i g_i\) and
  \(h=\lim_ih_i\) in \(S_\infty\) by Claim~\ref{cl:E d X is generically E
    S infty Y-ergodic}-(ix),(x), and so
  \(g\cdot\theta(x)=h\cdot\theta(y)\) by Claim~\ref{cl:E d X is
    generically E S infty Y-ergodic}-(i)--(vi), proving Theorem~\ref{t:E_d^X is generically E_S infty Y-ergodic}.
  
  The construction of the sequences of Claim~\ref{cl:E d X is
    generically E S infty Y-ergodic} is made by induction on
  \(i\in\N\). Let \(x_0=x\), \(U_0=X\), \(n_0=0\) and \(g_0=h_0=1_{S_\infty}\),
  and choose \(V_0\) and \(k_0\) so that \(y\in V_0\) and
  \[
  (V_0\cap C)\times\{0\}\subseteq I^{-1}(k_0).
  \]
  Suppose that all the members with
  indices \(\le i\in \N\) of these sequences have been constructed. Then
   \(x_{i+1}\), \(g_{i+1}\) and \(U_{i+1}\) are constructed in the following
  manner. (The constructions of \(y_{i+1}\), \(h_{i+1}\) and
  \(V_{i+1}\) are analogous.)
  
  Let \(U\subseteq V_i\) be a non-empty open set such that
  \(\QQ(y_i,V_i,-k_i)\cap U\) is dense in \(U\), and such that
  \(\diam(\theta(U\cap C))<2^{-i-1}\) (which is possible because
  \(\theta\) is continuous on \(C_0\)). Choose
  \(x_{i+1}\in\QQ(x_i,U_i,-n_i)\cap U\), and take \(z_0,\dots,z_k\in
  U_i\) so that \(z_0=x_i\), \(z_k=x_{i+1}\) and \(z_a\in
  E_{-n_i}(z_{a-1})\) for \(a\in\{1,\dots,k\}\). You may assume that
  \(i>0\) because~(ix) does not restrict the choice of \(g_1\).
  
  \begin{claim}\label{cl:z a in C}
     We can assume that \(z_a\in C\) for all \(a\in\{0,\dots,k\}\).
  \end{claim}
  
  Claim~\ref{cl:z a in C} is proved by showing that for each
  \(a\in\{0,\dots,k\}\) there exists
  \[
  z'_a\in U_i\cap(E_{-n_i}^{k-a})^{-1}(P_U)\cap C
  \]
  so that \(z'_0=x_i\) and, for \(a\in\{1,\dots,k\}\), \(z'_a\in
  E_{-n_i}(z'_{a-1})\); then we can choose
  \(x'_{i+1}=z'_k\) instead of \(x_{i+1}\), and \(z'_a\) instead of
  \(z_a\). We have
  \[
  z'_0=x_i\in U_i\cap (E_{-n_i}^k)^{-1}(P_U)\cap C.
  \]
  Suppose that \(z'_a\) has been constructed for \(a<k\). Since
  \(z'_a\in C\) and \(E_{-n_i}^{k-a-1}\) is continuous by Lemma~\ref{l:F
    circ E is continuous}-(i), the set
  \[
  E_{n_i}(z'_a)\cap U_i\cap(E_{-n_i}^{k-a-1})^{-1}(P_U)\cap C
  \]
  is residual in  \(
  E_{n_i}(z'_a)\cap U_i\cap(E_{-n_i}^{k-a-1})^{-1}(P_U)\). 
  So, by Remark~\ref{r:type I}-(i), there is
  \[
  z'_{a+1}\in E_{n_i}(z'_a)\cap
  U_i\cap(E_{-n_i}^{k-a-1})^{-1}(P_U)\cap C,
  \]
  as desired for the proof of Claim~\ref{cl:z a in C}.
  
  Continuing with the proof of Claim~\ref{cl:E d X is generically E S
    infty Y-ergodic}, Claim~\ref{cl:z a in C} gives \(I(z_a,N_i)=n_i\)
  for all \(a\in\{0,\dots,k\}\) by the induction hypothesis with
  Claim~\ref{cl:E d X is generically E S infty Y-ergodic}-(vii).
  
  \begin{claim}\label{cl:f a}
    We can assume that, for each \(a<k\), there exists some \(f_a\in
    W_{N_i}\) such that \(f_a\cdot\theta(z_a)=\theta(z_{a+1})\).
  \end{claim}
  
  As  in Claim~\ref{cl:z a in C}, we show that the condition of this
  claim is satisfied by a new finite sequence of points
  \[
  z'_a\in U_i\cap(E_{-n_i}^{k-a})^{-1}(P_U)\cap C
  \]
  so that \(z'_0=x_i\) and, for \(a\in\{1,\dots,k\}\), 
  \(z'_a\in E_{-n_i}(z'_{a-1})\); in particular, \(I(z'_a,N_i)=n_i\) as above. This
  new sequence is constructed by induction on \(a\). First, let
  \(z'_0=x_i\), and suppose that \(z'_a\) was constructed for all 
  \(a<k\). Since \(I(z'_a,N_i)=n_i\), \(\forall^*z\in E_{-n_i}(z'_a)\),
  \(\exists f\in W_{N_i}\) so that \(f\cdot\theta(z_a)=\theta(z)\). So the
  set of points
  \[
  z\in E_{-n_i}(z'_a)\cap U_i\cap(E_{-n_i}^{k-a-1})^{-1}(P_U)\cap C
  \]
  such that \(\exists f\in W_{N_i}\) so that
  \(f\cdot\theta(z_a)=\theta(z)\) is residual in
  \[ E_{-n_i}(z'_a)\cap U_i\cap(E_{-n_i}^{k-a-1})^{-1}(P_U)\cap C.\]
  Hence \(f_a\cdot\theta(z'_a)=\theta(z'_{a+1})\) for some \(f_a\in
  W_{N_i}\) and some
  \[ z'_{a+1}\in E_{-n_i}(z'_a)\cap
  U_i\cap(E_{-n_i}^{k-a-1})^{-1}(P_U)\cap C
  \] by Remark~\ref{r:type I}-(i), completing the proof of
  Claim~\ref{cl:f a}.
  
  According to Claim~\ref{cl:f a},
  \(f^*_i\cdot\theta(x_i)=\theta(x_{i+1})\) for \(f^*_i=f_{k-1}\cdots
  f_0\in W_{N_i}\). Then let \(g_{i+1}=f^*_ig_i\). Moreover we can
  take some open neighborhood \(U_{i+1}\) of \(x_{i+1}\) in \(U\) and
  some \(n_{i+1}\in\N\) such that \(\diam(\theta(U_{i+1}\cap
  C))<2^{-i-1}\) and
    \[ (U_{i+1}\cap C)\times\{N_{i+1}\}\subseteq I^{-1}(n_{i+1}),
    \] where \(N_{i+1}\) is defined according Claim~\ref{cl:E d X is
      generically E S infty Y-ergodic}-(vii). These choices of
    \(x_{i+1}\), \(g_{i+1}\), \(U_{i+1}\) and \(n_{i+1}\) satisfy the
    conditions of Claim~\ref{cl:E d X is generically E S infty
      Y-ergodic}.
\end{proof}

\begin{rem} The proofs of Lemma~\ref{l:theta} and
    Theorem~\ref{t:E_d^X is generically E_S infty Y-ergodic} are
    directly inspired by those of~\cite[Lemma~3.17 and
    Theorem~3.18]{Hjorth2000}.
\end{rem}

\section{A class of turbulent metric relations}\label{s:U_R,r}

Let \(X\) be a set, and let \(\, \mathcal{U}=\{U_{R,r}\subseteq
X\times X\mid R,r>0\,\}\) be a set of relations over \(X\) that
satisfy the following hypothesis. 

\begin{hyp}\label{h:U_R,r}
    \begin{enumerate}[(i)]
  
      \item \(\bigcap_{R,r>0}U_{R,r}=\Delta_X\);
      
      \item each \(U_{R,r}\) is symmetric;
      
      \item if \(R\le S\), then \(U_{R,r}\supseteq U_{S,r}\) for all \(r>0\);
    
      \item \(U_{R,r}=\bigcup_{s<r}U_{R,s}\) for all \(R,r>0\); and
    
      \item there is some function \(\phi:(\R_+)^2\to\R_+\) such that,
for all \(R,S,r,s>0\),
        \begin{gather*} R\le\phi(R,r),\\ 
        (R\le S,\ r\le s)\Longrightarrow\phi(R,r)\le\phi(S,s),\\ 
U_{\phi(R,r+s),r}\circ
U_{\phi(R,r+s),s}\subseteq U_{R,r+s}.
        \end{gather*}
  
    \end{enumerate}
\end{hyp}

By Hypothesis~\ref{h:U_R,r}, the sets \(U_{R,r}\) form a base of
entourages of a Hausdorff uniformity, also denoted by \(\mathcal{U}\),
on \(X\). This uniformity is metrizable because the entourages
\(U_{n,1/n}\), \(n\in\Z_+\), form a countable base for it.

For each \(r>0\), let \(E_r=\bigcap_{R>0}U_{R,r}\). This set is symmetric by
Hypothesis~\ref{h:U_R,r}-(ii); moreover
  \begin{equation}\label{E_s circ E_r subset E_r+s} 
    E_s\circ E_r\subseteq E_{r+s},
  \end{equation} for \(r,s>0\), by Hypothesis~\ref{h:U_R,r}-(v).

\begin{lemma}\label{l:U_S,rsubset Int(U_R,r)} 
  For \(R,r>0\) and \(S=\phi(\phi(R,r),r)\) \textup{(}where \(\phi\)
  is the function given in
  Hypothesis~\ref{h:U_R,r}-\textup{(}v\textup{))}, the set
  \(U_{S,r}\subseteq\Int(U_{R,r})\).
\end{lemma}

\begin{proof} 
  Let \((x,y)\in U_{S,r}\). By Hypothesis~\ref{h:U_R,r}-(iv), there is
  some \(r_0 <r\) such that \((x,y)\in U_{S,r_0}\). Let
  \(r_1=\dfrac{r-r_0}{2}\). By Hypothesis~\ref{h:U_R,r}-(v),
  \begin{align*} 
    U_{S,r_1}\circ U_{S,r_0}\circ U_{S,r_1}
    &\subseteq U_{\phi(\phi(R,r),\frac{r+r_0}{2}),r_1}\circ
    U_{\phi(\phi(R,r),\frac{r+r_0}{2}),r_0}\circ U_{\phi(R,r),r_1}\\
    &\subseteq U_{\phi(R,r),\frac{r+r_0}{2}}\circ
    U_{\phi(R,r),r_1}\subseteq U_{R,r}.
  \end{align*} 
  So, by Hypothesis~\ref{h:U_R,r}-(ii),
  \(U_{S,r_1}(x)\times U_{S,r_1}(y)\subseteq U_{R,r}\), which implies that
  \((x,y)\in\Int(U_{R,r})\).
\end{proof}

\begin{cor}\label{c:E_r=bigcap_R>0 Int(U_R,r)}
  For each \( r>0\), the set \(E_r=\bigcap_{R>0}\Int(U_{R,r})\).
\end{cor}

Hypothesis~\ref{h:U_R,r}-(iii) and Corollary~\ref{c:E_r=bigcap_R>0
  Int(U_R,r)} imply that \(E_r=\bigcap_{n=1}^\infty\Int(U_{n,r})\) for
all \(r>0\) and so \(E_r\) is a \(G_\delta\) subset of \(X\times X\). Hence
the relations \(E_r\) satisfy Proposition~\ref{l: type I, with EE'}-(a'),(b').

Let \(d:X\times X\to[0,\infty]\) be defined by
\begin{equation}\label{d} 
  d(x,y)=\inf\{\,r>0\mid (x,y)\in E_r\,\},
\end{equation} 
where \(\inf \emptyset=\infty\), so \(d(x,y)=\infty\) if
\(x\notin \bigcup_{r>0} E_r(y)\). It easily follows from Hypothesis~\ref{h:U_R,r}
that \(d\) is a metric relation over \(X\). Note also
that, for \(0<r<s\),
  \[ 
    B_d(x,r)\subseteq E_r(x)\subseteq B_d(x,s),
  \] 
 and therefore
  \begin{equation}\label{E_d^X=bigcup_r>0 E_r}
    E_d^X=\bigcup_{r>0}E_r,
  \end{equation} 
and \(B_d(x,r)\subseteq U_{R,r}(x)\) for all \(R,r>0\) and all \(x\in X\), which
implies that the topology induced by \(d\) on \(X\) is finer than the
topology induced by the uniformity \(\mathcal{U}\) on
\(X\). Consequently, \(d\) satisfies conditions~(ii) of Definition~\ref{d:type I} and Proposition~\ref{l: type I, with EE'}-(c'),(d')
 with the relations \(E_r\).

\begin{example}\label{ex:d_R} 
  Let \(d_R\),  \(R>0\), be pseudo-metrics on a set, \(X\),
  such that
  \begin{gather} 
  R\le S\Longrightarrow d_R\le d_S,\label{d_R le d_S}\\ 
      (\,\forall R>0,\ d_R(x,y)=0\,)\Longrightarrow
    x=y.\label{x=y}
  \end{gather} Then the sets
  \[ U_{R,r}=\{\,(x,y)\in X\times X\mid d_R(x,y)<r\,\}
  \] 
  satisfy Hypothesis~\ref{h:U_R,r}; in particular,
  Hypothesis~\ref{h:U_R,r}-(v) holds with \(\phi(R,r)=R\) since
  the triangle inequality of each \(d_R\) and~\eqref{d_R le d_S} give
  \begin{equation}\label{U_R,r circ U_S,s subset U_min R,S,r+s}
    U_{R,r}\circ U_{S,s}\subseteq U_{\min\{R,S\},r+s},
  \end{equation} 
  for all \(R,S,r,s>0\). It follows that \(U_{R,r}(x)\) is open for all
  \(x\in X\) and all \(R,r>0\). In this case, the relations \(U_{R,r}\) induce
  the topology defined by the set of pseudo-metrics \(\{d_R\}\), and the
  corresponding sets \(E_r\) define the metric relation
  \(d=\sup_{R>0}d_R\).
\end{example}

For \(d\) (the metric equivalence relation given by \eqref{d})
satisfies the remaining conditions of Definition~\ref{d:type I},
further hypothesis are required.

\begin{hyp}\label{h:type I}
  \begin{enumerate}[(i)]
      
  \item \(X\) is a Polish space (with the topology induced by the uniformity
    \(\mathcal{U}\));
      
    \item for all \(R,r,s>0\) and \(x\in X\), if \(y\in E_s(x)\), then there are some
      \(T,t>0\) such that \(U_{T,t}(y)\subseteq E_s\circ U_{R,r}(x)\); and,
      
    \item for all \(r,s>0\) and \(x\in X\), if  \(y\in E_s(x)\)
      and \(V\) is a neighborhood of \(y\) in \(X\), then there is a
      neighborhood \(W\) of \(y\) in \(X\) such that
      \[
      E_r(W)\cap E_r(E_s(x))\subseteq E_r(V\cap E_s(x)).
      \]
      
    \end{enumerate}
 \end{hyp}

\begin{prop}\label{p:d is of type I} 
  If \(\mathcal{U}\) satisfies Hypothesis~2, then \(d\) is of type~I.
\end{prop}

\begin{proof} It only remains to show that \(d\) satisfies Proposition~\ref{l: type I, with EE'}-(e'),(f'). 
  
  Hypothesis~\ref{h:type I}-(ii) simply means that \(E_s\) is open
  and hence  continuous because it is symmetric.
  
  Let \(r,s,t>0\), \(x\in X\) and \(y\in E_s(x)\). Suppose that
  \(E_r\circ E_s\supseteq E_t\), and let \(V\) be a neighborhood of \(y\) in
  \(X\). By Hypothesis~\ref{h:type I}-(iii), there is some open
  neighborhood \(W\) of \(y\) in \(X\) such that
  \[ E_r(W)\cap E_t(x)\subseteq E_r(W)\cap E_r(E_s(x))\subseteq E_r(V\cap
  E_s(x)).
  \] 
  Since \(E_r(W)\) is open in \(X\), this proves that
  \(E_r\cap(E_s(x)\times E_t(x))\) is an open relation over \(E_s(x)\) and
  \(E_t(x)\).
\end{proof}

\begin{rem}\label{r:E_s is uniformly continuous} 
  In some applications, the following condition, which is stronger
  than Hypothesis~\ref{h:type I}-(ii), is satisfied: for all
  \(R,r,s>0\), there are some \(T,t>0\) such that \(U_{T,t}\circ E_s\subseteq
  E_s\circ U_{R,r}\). This means that each \(E_s\) is ``uniformly open''
  (or ``uniformly continuous,'' because it is symmetric).
\end{rem}

The metric equivalence relation \(d\) will be shown to be turbulent
under the following additional hypothesis.

\begin{hyp}\label{h:turbulent}
  \begin{enumerate}[(i)]
    
  \item \(E_d^X\) has more than one equivalence class;
      
  \item for each \(x\), \(y\) in \(X\) and each \(R,r>0\), there is
    \(s>0\) such that \(U_{R,r}(x)\cap E_s(y)\neq\emptyset\); and
      
  \item for each \(x\in X\) and each \(R,r>0\), there are \(S,s>0\), a
    dense subset \(\DD\subseteq U_{S,s}(x)\cap E_d^X(x)\), and a
    \(d\)-dense subset of \(\DD\) such that every pair of points in
    \(\DD\) can be joined by a \(d\)-continuous path in
    \(U_{R,r}(x)\).
    
  \end{enumerate}
\end{hyp}

\begin{lemma}\label{l:E_d^X is minimal}
  The relation \(E_d^X\) is minimal.
\end{lemma}

\begin{proof} 
  This follows from Hypothesis~\ref{h:turbulent}-(ii)
  and~\eqref{E_d^X=bigcup_r>0 E_r}.
\end{proof}

\begin{lemma}\label{l:overline E_r(x) subset E_s(x)} 
  If \(r<s\), then, for all \(x\in X\), \(\overline{E_r(x)}\subseteq E_s(x)\).
\end{lemma}

\begin{proof} 
  If \(y\in\overline{E_r(x)}\) and \(R>0\), then
  \(U_{\phi(R,s),s-r}(y)\cap U_{\phi(R,s),r}(x)\neq\emptyset\). So
  \(y\in\bigcap_{R>0}U_{R,s}=E_s(x)\) by
  Hypothesis~\ref{h:U_R,r}-(ii),(v).
\end{proof}
    
\begin{lemma}\label{l:Int(E_r(x))=emptyset} 
  For all \(x\in X\) and \(r>0\), \(\Int(E_r(x))=\emptyset\).
\end{lemma}

\begin{proof} Suppose that \(\Int(E_r(x))\neq\emptyset\). Then, for each
  \(y\in X\), the intersection \(E_s(y)\cap E_r(x)\neq\emptyset\) for some
  \(s>0\), by Lemma~\ref{l:E_d^X is minimal} and~\eqref{E_d^X=bigcup_r>0
    E_r}. Therefore \(y\in E_{r+s}(x)\) by~\eqref{E_s circ E_r subset
    E_r+s}. So \(X=E_d^X(x)\) by~\eqref{E_d^X=bigcup_r>0
    E_r}, contradicting Hypothesis~\ref{h:turbulent}-(i).
\end{proof}

\begin{prop}\label{p:d is turbulent} 
The relation \( E_d^X\) is turbulent.
\end{prop}

\begin{proof} The relation \(E_d^X\) is minimal by Lemma~\ref{l:E_d^X
    is minimal}.  Each equivalence class of \(E_d^X\) is meager by
  Lemmas~\ref{l:overline E_r(x) subset E_s(x)}
  and~\ref{l:Int(E_r(x))=emptyset} and~\eqref{E_d^X=bigcup_r>0
    E_r}. Finally, the local equivalence classes of \(E_d^X\) are
  somewhere dense because of Hypothesis~\ref{h:turbulent}-(iii).
\end{proof}

Theorem~\ref{t:E_d^X is generically E_S infty Y-ergodic}, and
Propositions~\ref{p:d is of type I} and~\ref{p:d is turbulent} have
the following immediate consequence.

\begin{prop}\label{p:E_d^X is generically E_S infty Y-ergodic} 
  For any Polish \(S_\infty\)-space \(Y\), the relation \(E_d^X\) is
  generically \(E_{S_\infty}^Y\)-ergodic.
\end{prop}
 

\begin{rem}\label{r: E_d^X not le_B E^Y_G}
  If we also assume that, for all \(r>0\) and residually many \(x,
  y\in X\), there exists \(s_0>0\) such that \(E_s(y)\setminus
  E_r(x)\) is dense in \(E_s(y)\) for all \( s>s_0\), then the proof
  of \cite[Theorem~8.2]{Hjorth2000} can be adapted to show that
  \(E_d^X\not\le_B E^Y_G\) for any Polish group \(G\) and any Polish
  \(G\)-space \(Y\). However the proof is not given because, in the
  applications, this is proved in \cite{AlvarezCandel:Non-reduction}.
\end{rem}

\section{The supremum metric relation}\label{s:supremum}

A concrete case of Example~\ref{ex:d_R} is \(C(\R)\), the space of real
valued continuous functions on \(\R\) endowed with the compact-open
topology, and the supremum metric relation, \(d_\infty\), which is
induced by the supremum norm, \(\|\ \|_\infty\), defined by
\(\|f\|_\infty=\sup_{x\in\R}|f(x)|\). For each \(R>0\), let \(d_R\) be the
pseudo-metric on \(C(\R)\) induced by the semi-norm \(\|\ \|_R\) given by
\(\|f\|_R=\sup_{|x|<R}|f(x)|\). Clearly, this set \(\{d_R\mid R>0\}\) of
pseudo-metrics satisfies the conditions~\eqref{d_R le d_S}
and~\eqref{x=y}, and induces the compact-open topology of
\(C(\R)\). Moreover \(d_\infty=\sup_{R>0}d_R\). In this case, each
\(U_{R,r}\) (respectively, \(E_r\)) consists of the pairs \((f,g)\) that
satisfy \(\|f-g\|_R<r\) (respectively, \(|f(x)-g(x)|<r\) for all
\(x\in\R\)).

Write \(E_\infty=E_{d_\infty}^{C(\R)}\), and
\(B_\infty(f,r)=B_{d_\infty}(f,r)\) for each \(f\in C(\R)\) and
\(r>0\). Then \(f E_\infty g\) if and only if \(f-g\) is bounded; in
particular, the bounded functions of \(C(\R)\) form an equivalence
class of \(E_\infty\).

Theorem~\ref{t: d_infty, d_GH and d_QI} for \((d_\infty,E_\infty)\)
follows from Propositions~\ref{p:d is of type I} and~\ref{p:d is
  turbulent}--\ref{p:E_d^X is generically E_S infty Y-ergodic} once
Hypotheses~\ref{h:U_R,r}--\ref{h:turbulent} are shown to hold.

\begin{rem}\label{r:d_infty} 
  Let \(C_b(\R)\subseteq C(\R)\) be the subset of bounded continuous
  functions. The sum of functions makes the space \(C(\R)\) into a
  Polish group, and \(C_b(\R)\) into a subgroup. The orbit relation of
  the action of \(C_b(\R)\) on \(C(\R)\) given by translation is
  \(E_\infty\) and there is no Polish topology on \(C_b(\R)\)
    with respect to which this action is continuous
    \cite{AlvarezCandel:Non-reduction}.
  
  For instance, consider the restriction of the compact-open topology
  to \(C_b(\R)\). Then the action of \(C_b(\R)\) on \(C(\R)\) is continuous,
  \(C_b(\R)\) is metrizable because \(C(\R)\) is completely metrizable,
  and \(C_b(\R)\) is separable because it contains \(C_0(\R)\), which is
  dense in \(C(\R)\) and separable (by the Stone-Weierstrass
  theorem). But \(C_b(\R)\) is not completely metrizable with the
  compact-open topology; in particular, it is not closed in \(C(\R)\).
  
  Consider now the topology on \(C_b(\R)\) induced by \(\|\
  \|_\infty\). Then the action of \(C_b(\R)\) on \(C(\R)\) is continuous,
  and \(C_b(\R)\) is completely metrizable; indeed, it is a Banach
  algebra with \(\|\ \|_\infty\). However \(C_b(\R)\) is not separable
  with \(\|\ \|_\infty\), which can be shown as follows. For each
  \(x\in\{\pm1\}^{\Z}\), let \(\tilde x\in C_b(\R)\) be the function whose
  graph is the union of segments between all consecutive points in the
  graph of \(x\). Then \(\{\,B_\infty(\tilde x,1)\mid
  x\in\{\pm1\}^{\Z}\,\}\) is an uncountable set of disjoint open
  subsets of \(C_b(\R)\). So \(C_b(\R)\) is not second countable, and
  therefore it is not separable.
\end{rem}

According to Example~\ref{ex:d_R}, the sets \(U_{R,r}\) satisfy
Hypothesis~\ref{h:U_R,r} and induce \(d_\infty\). In this case, the
inclusion~\eqref{E_s circ E_r subset E_r+s} becomes the equality
\begin{equation}\label{E_r circ E_s=E_r+s}
  E_r\circ E_s=E_{r+s}
\end{equation} for all \(r,s>0\); this holds because, if \(g\in
E_{r+s}(f)\), then
\[ 
  f+\frac{s}{r+s}(g-f)\in E_r(g)\cap E_s(f).
\]

It is well known that \(C(\R)\) is Polish (Hypothesis~\ref{h:type
  I}-(i)). The following lemma shows that Hypothesis~\ref{h:type
  I}-(ii) is satisfied in this case.

\begin{lemma}\label{l:U_R,r circ E_s=E_s circ U_R,r=U_R,r+s}
  For all \(R,r,s>0\), \(U_{R,r}\circ E_s=E_s\circ U_{R,r}=U_{R,r+s}\).
\end{lemma}

\begin{proof} 
  If \(S\ge R\), then, for all \(f,g,h\in C(\R)\),
    \[
      d_R(f,h) \le d_R(f,g)+ d_R(g,h) \le d_R(f,g) + d_S(g,h),
    \]
  because \(d_R\le d_S\). This implies that
\(U_{R,r}\circ U_{S,s}\) and \(U_{S,s}\circ U_{R,r}\) are both contained in
\(U_{R,r+s}\), which in turn implies that  \(U_{R,r}\circ E_s\) and
\(E_s\circ U_{R,r}\) are both contained in \(U_{R,r+s}\).
   
To prove the reverse inclusions, let \( f\in C(\R)\) and \(g\in
U_{R,r+s}(f)\). Then
\begin{align*}
  h_0&=f+\frac{s}{r+s}(g-f)\in U_{R,s}(f)\cap U_{R,r}(g),\\
  h_1&=f+\frac{r}{r+s}(g-f)\in U_{R,r}(f)\cap U_{R,s}(g).
\end{align*} 
By continuity, \(h_0\in U_{S,s}(f)\) and \(h_1\in U_{S,s}(g)\) for some
\(S>R\). Let \(\lambda:\R\to[0,1]\) be any continuous function such that
\(\supp \lambda \subseteq[-S,S]\) and \(\lambda\equiv1\) on \([-R,R]\). Then
\begin{align*} 
f+\lambda(h_0-f)&\in E_s(f)\cap U_{R,r}(g),\\
  g+\lambda(h_1-g)&\in U_{R,r}(f)\cap E_s(g),
\end{align*}
which implies that \(g\in(U_{R,r}\circ E_s)(f)\cap(E_s\circ
U_{R,r})(f)\).
\end{proof}

\begin{cor} 
  If \(R, S, r, s>0\), then \(U_{R,r}\circ  U_{S,s}=U_{\min\{R,S\},r+s}\).
\end{cor}

\begin{proof} The inclusion ``\(\subseteq\)'' is~\eqref{U_R,r circ U_S,s
    subset U_min R,S,r+s}, and ``\(\supseteq\)''
  follows from Lemma~\ref{l:U_R,r circ E_s=E_s circ U_R,r=U_R,r+s}.
\end{proof}

\begin{lemma}\label{l:E_r cap(E_s(f) times E_t(f)), d_infty} 
If
  \(T,r,s,t>0\), \(f\in C(\R)\) and \(g\in E_s(f)\) are such that
  \(U_{T,t'}(g)\subseteq U_{T,s}(f)\) for some \(t'>t\), then
        \begin{equation}\label{E_r cap(E_s(f) times E_t(f)), d_infty}
          U_{T,t+r}(g)\cap E_{r+s}(f)=E_r(U_{T,t}(g)\cap E_s(f)).
        \end{equation}
\end{lemma}

\begin{proof} 
The inclusion ``\(\supseteq\)'' follows from~\eqref{E_s circ E_r subset
  E_r+s} and Lemma~\ref{l:U_R,r circ E_s=E_s circ U_R,r=U_R,r+s}. To
prove ``\(\subseteq\)'', let \(h\in U_{T,t+r}(g)\cap
E_{r+s}(f)\). By~\eqref{E_r circ E_s=E_r+s} and Lemma~\ref{l:U_R,r circ
  E_s=E_s circ U_R,r=U_R,r+s}, there are \(g_0\in E_r(h)\cap
U_{T,t}(g)\) and \(f_0\in E_r(h)\cap E_s(f)\). By continuity, \(g_0\in
U_{T',t}(g)\subseteq U_{T',s}(f)\) for some \(T'>T\). Let
\(\lambda:\R\to[0,1]\) be any continuous function such that
\(\supp \lambda \subseteq [-T',T']\) and \(\lambda\equiv1\) on \([-T,T]\). Then
    \[ f_0+\lambda(g_0-f_0)\in E_r(h)\cap U_{T,t}(g)\cap E_s(f),
    \] and so \(h\in E_r(U_{T,t}(g)\cap E_s(f))\).
\end{proof}

\begin{cor}
The supremum metric relation \(d_\infty\) satisfies
Hypothesis~\ref{h:type I}-(iii).
\end{cor}

\begin{proof}
  Let \(r,s>0\), \(f\in C(\R)\), \(g\in E_s(f)\), and \(V\) a
  neighborhood \(g\) in \(C(\R)\). Since \(V\) can be chosen as small
  as desired, we can assume that \(V=U_{T,t}(g)\) for some \(T,t>0\).
  Since \(g\in E_s(f)\subseteq U_{T,s}(f)\), there is \(t'>0\) such that
  \(U_{T,t'}(g)\subseteq U_{T,s}(f)\), and we can also suppose that
  \(t<t'\), obtaining~\eqref{E_r cap(E_s(f) times E_t(f)), d_infty} by
  Lemma~\ref{l:E_r cap(E_s(f) times E_t(f)), d_infty}. But~\eqref{E_r
    cap(E_s(f) times E_t(f)), d_infty} gives the inclusion of
  Hypothesis~\ref{h:type I}-(iii) for \(W=V\) by~\eqref{E_r circ
    E_s=E_r+s} and Lemma~\ref{l:U_R,r circ E_s=E_s circ
    U_R,r=U_R,r+s}.
\end{proof}

The fact that \(E_\infty\) has more than one class
(Hypothesis~\ref{h:turbulent}-(i)) is obvious because
\(d_\infty(f,g)=\infty\) if \(f\) is bounded and \(g\)
unbounded. Hypotheses~\ref{h:turbulent}-(ii),(iii) is a consequence of
the following lemmas.

\begin{lemma} 
  For every \(f,g\in C(\R)\) and every \(R,r>0\), if \(s>d_{R'}(f,g)\) for some
  \(R'>R\), then \(U_{R,r}(f)\cap E_s(g)\neq\emptyset\).
\end{lemma}

\begin{proof} Let \(\lambda:\R\to[0,1]\) be a continuous function such
  that \(\supp \lambda \subseteq [-R',R']\) and \(\lambda\equiv1\) on
  \([-R,R]\). Then \(g+\lambda(f-g)\in U_{R,r}(f)\cap E_s(g)\).
\end{proof}

\begin{lemma} For every \(R,r>0\) and every \(f\in C(\R)\), the set
\(U_{R,r}(f)\cap E_\infty(f)\) is \(d_\infty\)-path connected.
\end{lemma}

\begin{proof} For every \(g\in U_{R,r}(f)\cap E_\infty(f)\), the mapping
\(t\mapsto tf+(1-t)g\) defines a \(d_\infty\)-continuous path in
\(U_{R,r}(f)\cap E_\infty(f)\) from \(g\) to \(f\).
\end{proof}

\begin{rem}\label{r:B_infty(f,r) is not Polish} The symmetric
  relations over \(C(\R)\) with fibers the balls \(B_\infty(f,r)\)
  cannot be used instead of the relations \(E_r\) to show that
  \(d_\infty\) is of type~I. For instance, each ball \(B_\infty(f,r)\)
  is not \(G_\delta\) in \(C(\R)\); otherwise it would be Polish, and
  therefore it would be a Baire space with the induced topology. But
  \(\emptyset\) is residual in \(B_\infty(f,r)\) for all \(r>0\), as
  the following argument shows. Let \((r_n)\) and \((R_n)\) be
  sequences such that \(0<r_n\uparrow r\) and
  \(0<R_n\uparrow\infty\). For each \(n\), let \(U_n\) be the set of
  functions \(g\in B_\infty(f,r)\) such that
    \[ 
      \sup_{|x|>R_n}|f(x)-g(x)|>r_n.
    \] The the sets \(U_n\) are open and dense in
\(B_\infty(f,r)\) and their intersection is empty.
\end{rem}

\section{The Gromov space}\label{s:Gromov sp}

In this section, we recall from
  \cite{AlvarezCandel:Non-reduction} some basic definitions and
  properties concerning the Gromov space, and also prove some new
  results.

Let \(M\) be a metric space and let \(d_M\), or simply \(d\), be its
distance function. The \textit{Hausdorff distance} between two
non-empty subsets, \(A,B\subseteq M\), is given by
  \[ 
  H_d(A,B)=\max\biggl\{\sup_{a\in A}\inf_{b\in B}d(a,b),
  \sup_{b\in B}\inf_{a\in A}d(a,b)\biggr\}.
  \] 
  Note that \(H_d(A,B)=H_d(\overline{A},\overline{B})\), and
  \(H_d(A,B)=0\) if and only if \(\overline{A}=\overline{B}\). Also,
  it is well known and easy to prove that \(H_d\) satisfies the
  triangle inequality, and that its restriction to the set of
  non-empty compact subsets of \(M\) is finite valued, and defines
  there a complete metric if \(M\) is complete.

Let \(M\) and \(N\) be arbitrary non-empty metric spaces. A metric on \(M\sqcup
N\) is called \textit{admissible} if its restrictions to \(M\) and \(N\)
are \(d_M\) and \(d_N\), where \(M\) and \(N\) are identified with their
canonical injections in \(M\sqcup N\). The \textit{Gromov-Hausdorff
distance} (or \textit{GH distance}) between \(M\) and \(N\) is defined by 
  \[ 
    d_{GH}(M,N)=\inf_dH_d(M,N),
  \] where the infimum is taken over all admissible metrics \(d\) on
\(M\sqcup N\).  It is well known that
\(d_{GH}(M,N)=d_{GH}(\overline{M},\overline{N})\), where \(\overline{M}\)
and \(\overline{N}\) denote the completions of \(M\) and \(N\),
\(d_{GH}(M,N)=0\) if \(\overline{M}\) and \(\overline{N}\) are
isometric, \(d_{GH}\) satisfies the triangle inequality, and
\(d_{GH}(M,N)<\infty\) if \(\overline{M}\) and \(\overline{N}\) are compact.

There is also a pointed version of \(d_{GH}\) which satisfies analogous
properties: the (\textit{pointed}) \textit{Gromov-Hausdorff distance}
(or \textit{GH distance}) between two pointed metric spaces, \((M,x)\)
and \((N,y)\), is defined by
  \begin{equation}\label{d_GH(M,x;N,y)} d_{GH}(M,x;N,y) =
\inf_d\max\{d(x,y),H_d(M,N)\},
\end{equation} where the infimum is taken over all admissible
metrics  \(d\) on
\(M\sqcup N\).

If \(X\) is any metric space and \(f:M\to X\) and \(g:N\to X\) are isometric
injections, then it is also well known that
\begin{align} 
  d_{GH}(M,N)&\le H_{d_X}(f(M),g(N)),\notag\\
  d_{GH}(M,x;N,y)&\le\max\{d_X(f(x),g(y)),H_{d_X}(f(M),g(N))\};\label{d_GH(M,x;N,y)
    le ...}
\end{align} 
indeed, these inequalities follow by considering, for each
\(\epsilon>0\), the unique admissible metric \(d_\epsilon\) on \(M\sqcup N\)
satisfying, for all \(u\in M\) and \(v\in N\),
  \[ 
  d_\epsilon(u,v)=d_X(f(u),g(v))+\epsilon.
  \]

 A metric space, or its distance function, is called \textit{proper} (or \textit{Heine-Borel})
 if every open ball has compact closure. This condition is equivalent
 to the compactness of the closed balls, which means that the distance
 function to a fixed point is a proper function. Any proper metric
 space is complete and locally compact, and its cardinality is not
 greater than the cardinality of the continuum. Therefore it may be assumed that their
 underlying sets are subsets of \(\R\). With this assumption, it
 makes sense to consider the set \(\MM_*\) of isometry classes, \([M,x]\),
 of pointed proper metric spaces, \((M,x)\). The set \(\MM_*\) is endowed
 with a topology introduced by M.~Gromov~\cite[Section~6]{Gromov1999},~\cite{Gromov1981}, which
 can be described as follows.

For a metric space \(X\), two subspaces, \(M,N\subseteq X\), two points,
\(x\in M\) and \(y\in N\), and a real number \(R>0\), let \(
H_{d_X,R}(M,x;N,y)\) be given by
  \[ H_{d_X,R}(M,x;N,y)=\max\biggl\{\sup_{u\in
B_M(x,R)}d_X(u,N),\sup_{v\in B_N(y,R)}d_X(v,M)\biggr\}.
  \] Then, for \(R,r>0\), let \(U_{R,r}\subseteq\MM_*\times\MM_*\) denote
the subset of pairs \(([M,x],[N,y])\) for which there is an admissible
metric, \(d\), on \(M\sqcup N\) so that 
  \[ 
    \max\{d(x,y),H_{d,R}(M,x;N,y)\}<r.
  \] 

The following lemma is obtained exactly like~\eqref{d_GH(M,x;N,y) le ...}.

\begin{lemma}\label{l:U_R,r} 
  For  \(([M,x],[N,y])\in\MM_*\times\MM_*\)
  to be in \(U_{R,r}\) it suffices that there exists a metric space,
  \(X\), and isometric injections, \(f:M\to X\) and \(g:N\to X\), such that
  \[
  \max\bigl\{d_X(f(x),g(y)),H_{d_X,R}(f(M),f(x);g(N),g(y))\bigr\}<r.
  \]
\end{lemma}

The following notation will be used: for a relation \(E\) on \(\MM_*\)
and \([M,x]\in \MM_*\), \(E([M,x])\) will be simply written as
\(E(M,x)\), and for a metric relation \(d\) on \(\MM_*\) and
\([M,x],[N,y]\in\MM_*\), \(d([M,x],[N,y])\) will be denote by
\(d(M,x;N,y)\).

The sets \(U_{R,r}\) satisfy Hypothesis~\ref{h:U_R,r}
  \cite[Lemma~2.1]{AlvarezCandel:Non-reduction}; in particular,
  Hypothesis~\ref{h:U_R,r}-(v) is satisfied as stated in the following
  lemma.

\begin{lemma}[{\cite[Lemma~2.1-(v)]{AlvarezCandel:Non-reduction}}]\label{l:U_S,r circ U_S,s subset U_R,r+s} 
  If \(R, r, s>0\), then \(U_{S,r}\circ U_{S,s}\subseteq U_{R,r+s}\), where
  \(S=R+2\max\{r,s\}\).
\end{lemma}

Since the sets \(U_{R,r}\) satisfy Hypothesis~\ref{h:U_R,r}, they form
a base of entourages for a metrizable uniformity on \(\MM_*\). Endowed
with the induced topology, \(\MM_*\) is what is called the
\textit{Gromov space}. It is well known that \(\MM_*\) is a Polish
space (see \textit{e.g.}\ Gromov~\cite{Gromov1999} or
Petersen~\cite{Petersen1998}); in particular, the set of the pointed
finite metric spaces with \(\Q\)-valued metrics is a countable dense
subset of \(\MM_*\).

Some relevant subspaces of \(\MM_*\) are defined by the following
classes of metric spaces: proper ultrametric spaces, proper length
spaces, connected complete Riemannian manifolds, connected locally
compact simplicial complexes, connected locally compact graphs and
finitely generated groups (via their Cayley graphs).

The following (generalized) dynamics can be considered on \(\MM_*\):
  \begin{description}
    
  \item[The canonical metric relation] The \textit{canonical
      partition} \(E_{\text{\rm can}}\) is defined by varying the
    distinguished point; \textit{i.e.}, as a relation, \(E_{\text{\rm can}}\)
    consists of all the pairs \(([M,x],[M,y])\), \(M\) a proper metric
    space \(M\) and \(x,y\in M\). There is a canonical map
    \(M\to\MM_*\), \(x\mapsto[M,x]\), which defines an embedding
    \(\Isom(M)\backslash M\to\MM_*\) whose image is \(E_{\text{\rm
        can}}(M,x)\) for any \(x\in M\).
    Note that \(\MM_*/E_{\text{\rm can}}\) can be identified to the
    set of isometry classes of proper metric spaces. 
    
  \item[The GH metric relation] It is defined by the pointed GH
    distance \(d_{GH}\). The notation \(E_{GH}=E^{\MM_*}_{d_{GH}}\) will
    be used. Since \(E_{\text{\rm can}}\subseteq E_{GH}\), the quotient
    set \(\MM_*/E_{GH}\) can be identified to the set of classes of
    proper metric spaces defined by the relation of being at finite GH
    distance.
    
\item[The Lipschitz metric relation] The \textit{Lipschitz
    partition}, \(E_{\text{\rm Lip}}\), is defined by the existence of
  pointed bi-Lipschitz bijections. It is induced by the
  \textit{Lipschitz metric relation}, \(d_{\text{\rm Lip}}\), which is
  defined by using the infimum of the logarithms of the dilation constants of
  bi-Lipschitz bijections.
        
\item[The QI metric relation]  The \textit{quasi-isometric partition}
  (or \textit{QI partition}), \(E_{QI}\), is the smallest equivalence
  relation over \(\MM_*\) that contains \(E_{GH}\cup E_{\text{\rm
      Lip}}\). It is induced by the \textit{quasi-isometric metric
    relation} (or \textit{QI relation}), \(d_{QI}\), defined as the
  largest metric relation over \(\MM_*\) smaller than both \(d_{GH}\) and
  \(d_{\text{\rm Lip}}\) (\textit{cf.}\ \cite[Lemma~6]{Plastria2009}). The
  quotient set \(\MM_*/E_{QI}\) can be identified to the set of
  quasi-isometry classes of proper metric spaces.
    
\item[The dilation flow] It is the multiplicative flow defined by
  \(\lambda\cdot[M,x]=[\lambda M,x]\), where \(\lambda M\) denotes \(M\)
  with its metric multiplied by \(\lambda\). This flow is used to define
  the asymptotic and tangent cones.
        
  \end{description}

In Sections~\ref{s:GH} and~\ref{s:QI}, we will study the GH and QI~metric
relations and prove Theorem~\ref{t: d_infty, d_GH and d_QI} for them. 
Some technical results and concepts related to the definition of
\(\MM_*\) to  be used in those sections are given
presently.

\begin{lemma}\label{l:d is proper} 
  Let \([M,x],[N,y]\in\MM_*\) and \(r>0\). If \(d\) is an admissible metric
  on \(M\sqcup N\) such that \(d(x,y)<r\) and \(H_d(M,N)<r\), then \(d\) is
  proper.
\end{lemma}

\begin{proof} 
  For every \(v\in N\),
  \[ 
    d_N(y,v)\le d(x,y)+d(x,v)<r+d(x,v)\,
  \] 
  and so
  \[ 
  B_d(x,R)\subseteq B_M(x,R)\sqcup B_N(y,R+r)
  \] 
  for all \(R>0\). The statement follows from this because \(M\) and \(N\) are
  proper.
\end{proof}

\begin{lemma}\label{l:d'}
  Let \([M,x],[N,y],[P,z]\in \MM_*\) and \(R,r>0\). Suppose that the
  pointed metric spaces \((B_P(z,R+2r),z)\) and \((B_N(y,R+2r),y)\)
  are isometric, and that there is an admissible metric, \(d\), on
  \(M\sqcup N\) such that \(d(x,y)<r\) and
  \(H_{d,R}(M,x;N,y)<r\). Then there exists a proper admissible
  metric, \(d'\), on \(M\sqcup P\) such that \(d'(x,z)<r\) and
  \(H_{d',R}(M,x;P,z)<r\).
\end{lemma}

\begin{proof}
  Let \(A=B_M(x, R+2r)\), \(B=B_N(y,R+2r)\) and \(C=B_P(z,R+2r)\), and let
  \(\phi:(B,y)\to(C,z)\) be an isometry. Let \(d'\) be the admissible
  metric on \(M\sqcup P\) satisfying, for \(u\in M\) and \(w\in P\),
  \[
  d'(u,w)=\inf\{\,d_M(u,u')+d(u',v)+d_P(\phi(v),w)\mid u'\in A,\ v\in B\,\}.
  \]
  Note that, for \(u\in A\) and \(v\in B\), \(d'(u,\phi(v))=d(u,v)\);
  in particular, \(d'(x,z)<r\).
  
  For each \(u\in B_M(x,R)\), there is \(v\in N\) such that
  \(d(u,v)<r\). Since
  \[
  d_N(y,v)\le d(y,x)+d_M(x,u)+d(u,v)<R+2r,
  \]
this \(v\in B_N(y,R+2r)\). So \(d'(u,\phi(v))=d(u,v)<r\), and therefore
  \(d'(u,P)<r\). Similarly, \(d'(w,M)<r\) for all \(w\in B_P(z,R)\),
  obtaining \(H_{d',R}(M,x;P,z)<r\).
  
  For each \(S>0\) and \(w\in P\cap B_{d'}(x,S)\), there is \(v\in B\)
  such that \(d(x,v)+d_P(\phi(v),w)<S\). So
    \[
      d_P(z,w)\le d_P(z,\phi(v))+d_P(\phi(v),w)<R+2r+S,
    \]
  obtaining
    \[
      B_{d'}(x,S)\subseteq B_M(x,S)\sqcup B_P(z,R+2r+S).
    \]
    Hence \(\overline{B_{d'}(x,S)}\) is compact since \(M\) and \(P\)
    are proper. This shows that \(d'\) is proper.
\end{proof}

\section{The GH metric relation}\label{s:GH}

For each \(r>0\), let \(E_r\subseteq\MM_*\times\MM_*\)
be the symmetric relation whose fibers are
\(E_r(M,x)=\bigcap_{R>0}U_{R,r}(M,x)\), where \(U_{R,r}(M,x)\) is as
defined in Section~\ref{s:Gromov sp}. The notation
\(B_{GH}(M,x;r)=B_{d_{GH}}([M,x],r)\) will be used.

\begin{lemma}\label{l:B_{GH}(M,x;r)}
  If \(0<r<s\), then
    \[
      B_{GH}(M,x;r)\subseteq E_r(M,x)\subseteq B_{GH}(M,x;s).
    \]
\end{lemma}

\begin{proof}
  The first inclusion is obvious. To verify the second one, let
  \([N,y]\) be any member of \(E_r(M,x)\). For each \(R>0\) there
  exists an admissible metric, \(d_R\), on \(M\sqcup N\) such that
  \(d_R(x,y)<r\) and \(H_{d_R,R}(M,x;N,y)<r\). Let \(\omega\) be a
  free ultrafilter on \([0,\infty)\). Then there is a unique
  admissible metric, \(d\), on \(M\sqcup N\) such that, for all \(u\in
  M\) and \(v\in N\),
  \[
  d(u,v)=\lim_{R\to\omega}d_R(u,v)+\frac{s-r}{2}.
  \]
   For each \(\epsilon>0\) there exists
  \(\Omega\in\omega\) such that,   for all \(R\in\Omega\),
  \[
    d(u,v)<d_R(u,v)+\frac{s-r}{2}+\epsilon,
  \]
 and thus
  \[
    d(x,y)\le d_R(x,y)+\frac{s-r}{2}+\epsilon<\frac{s+r}{2}+\epsilon.
  \]
  Because this holds for each \( \epsilon >0\), 
  \[
    d(x,y)\le\frac{s+r}{2}<s.
  \]

  Next, for every \(u\in M\), if \(R\in\Omega\) is \(>d(x,u)\), then
  \(d_R(u,N)<r\), and so \(d(u,N)<s\) as before. Similarly, \(d(v,M)<s\) for
  all \(v\in N\). Therefore \(H_d(M,N)<s\).
\end{proof}

\begin{cor}\label{c:d_GH}
  The metric relation over \(\MM_*\) defined by the sets \(U_{R,r}\) is
  \(d_{GH}\).
\end{cor}

By Propositions~\ref{p:d is of type I},~\ref{p:d is
  turbulent} and~\ref{p:E_d^X is generically E_S infty Y-ergodic}, and
Corollary~\ref{c:d_GH}, the case of \((d_{GH},E_{GH})\) in
Theorem~\ref{t: d_infty, d_GH and d_QI} follows by showing that the
sets \(U_{R,r}\) also satisfy Hypotheses~\ref{h:type
  I} and~\ref{h:turbulent}. It was already noted that \(\MM_*\) is Polish
(Hypothesis~\ref{h:type I}-(i)).

\begin{lemma}\label{l:U_S,s circ U_R,r subset E_s circ U_R,r}
 If \(R,r,s>0\), then  \(U_{R+2r+s,s}\circ U_{R,r}\subseteq E_s\circ
 U_{R,r}\).
\end{lemma}

\begin{proof} Let \(S=R+2r+s\).  If \([M,x]\in\MM_*\) and \([N,y]\in
  U_{S,s}\circ U_{R,r}(M,x)\), then there is \([P,z]\in
  U_{R,r}(M,x)\cap U_{S,s}(N,y)\). This means that there are
  admissible metrics, \(d\) on \(M\sqcup P\) and \(\bar d\) on \(N\sqcup P\),
  such that \(d(x,z)<r\), \(H_{d,R}(M,x;P,z)<r\), \(\bar d(y,z)<s\) and
  \(H_{\bar d,S}(N,y;P,z)<s\). Moreover, because of Lemma~\ref{l:d'},
  \(\bar d\) may be assumed to be a proper metric. The subset
  \[
  P'=(N\setminus B_N(y,S))\sqcup\overline{B_P(z,S)}\subseteq N\sqcup P
  \]
  is closed and so it becomes a proper metric space when endowed with
  the metric induced by \(\bar d\).
    
  \begin{claim}\label{cl:d_GH(N,y;P',z)<s}
    The pointed metric space \((P',z)\) satisfies \(d_{GH}(N,y;P',z)<s\).
  \end{claim}
    
  Since \(N\setminus P'\subseteq B_N(y,S)\) and \(P'\setminus
  N=\overline{B_P(z,S)}\), the Hausdorff distance
  \begin{multline*}
    H_{\bar d}(N,P')
    =\max\biggl\{\sup_{v\in B_N(y,S)}\bar d(v,P'),\sup_{w\in B_P(z,S)}\bar d(w,N)\biggr\}\\
    \le H_{\bar d,S}(N,y;P,z)<s,
  \end{multline*}
  and so Claim~\ref{cl:d_GH(N,y;P',z)<s} follows
  from~\eqref{d_GH(M,x;N,y) le ...}.
    
  From Claim~\ref{cl:d_GH(N,y;P',z)<s} and Corollary~\ref{c:d_GH}, it
  follows that \([P',z]\in E_s(N,y)\).
    
  \begin{claim}\label{cl:B_M'(y,R'+2epsilon)=B_N(y,R'+2epsilon)}
    \(B_{P'}(z,R+2r)=B_P(z,R+2r)\).
  \end{claim}
    
  The inclusion ``\(\supseteq\)'' of this identity is obviously true. To
  prove that the reverse inclusion ``\(\subseteq\)'' is also true, it
  suffices to note that \(B_{P'}(z,R+2r)\cap N=\emptyset\), which is
  true because, if there is \(v\in B_{P'}(z,R+2r)\cap N\), then
  \[
  d_N(y,v)\le\bar d(y,z)+\bar d(z,v)<s+R+2r=S,
  \]
  which contradicts that \(B_N(y,S)\cap P'=\emptyset\).
    
  From Claim~\ref{cl:B_M'(y,R'+2epsilon)=B_N(y,R'+2epsilon)} and
  Lemma~\ref{l:d'}, it follows that \([P',z]\in U_{R,r}(M,x)\). Hence
  \([N,y]\in E_s\circ U_{R,r}(M,x)\).
\end{proof}

A subset \(A\) of a metric space \(X\) is called a \textit{net} (as
defined in~\cite[Definition~2.14]{Gromov1999}) if there
is \(\epsilon>0\) such that \(d_X(u,A)\le\epsilon\) for all \(u\in X\), and
it is called \textit{separated} if there is \(\delta>0\) such that
\(d_X(a,b)\ge\delta\) for all \(a, b \in A\) with \(a\neq b\); the terms
\(\epsilon\)-\textit{net} and \(\delta\)-\textit{separated} are also used
in these cases. 

A separated subset of a metric space is discrete and therefore
closed. Hence, every separated subset of a proper metric space 
is a proper metric space when endowed with the induced metric.

If \(A\subseteq X\) is an \(\epsilon\)-net of a metric space,
\((X,d_X)\), then \(H_{d_X}(X,A)\le\epsilon\). So, if \(A\) is endowed
with the induced metric from \((X,d_X)\), then
\(d_{GH}(X,x;A,x)\le\epsilon\) for every \(x\in A\)
by~\eqref{d_GH(M,x;N,y) le ...}. Therefore, if in addition \(X\) is
proper and \(A\) is separated, then, for any \(\delta>\epsilon\), 
\([A,x]\in E_\delta(X,x)\) by Lemma~\ref{l:B_{GH}(M,x;r)}.

\begin{lemma}\label{l:separated nets}
  Let \(\epsilon>0\). For each metric space \(M\) and each
  \(\epsilon\)-separated subset \(S\subseteq M\), there exists an
  \(\epsilon\)-separated \(\epsilon\)-net of \(M\) that contains \(S\).
\end{lemma}

\begin{proof}
  By Zorn's lemma, the set of \(\epsilon\)-separated subsets of \(M\)
  that contain \(S\), ordered by inclusion, has a maximal element. It is
  easily checked that that maximal element is an \(\epsilon\)-net.
\end{proof}

The following is a ``converse'' to Lemma~\ref{l:U_S,r circ U_S,s
  subset U_R,r+s}.

\begin{lemma}\label{l:U_R,r+s subset U_R,s circ U_R,r}
If \(R, r, s > 0 \), then  \(U_{R,r+s}\subseteq U_{R,s}\circ U_{R,r}\).
\end{lemma}

\begin{proof}
  Let \([M,x]\in\MM_*\) and \([N,y]\in U_{R,r+s}(M,x)\). Then there is
  an admissible metric, \(d\), on \(M\sqcup N\) such that
  \(d(x,y)<r_0+s_0\) and \(H_{d,R}(M,x;N,y)<r_0+s_0\) for some
  \(r_0\in(0,r)\) and \(s_0\in(0,s)\). By Lemma~\ref{l:d'}, \(d\) may
  be assumed to be a proper metric.
   
   Let \(\epsilon>0\) be such that \(r_0+2\epsilon<r\) and
   \(s_0+2\epsilon<s\). By Lemma~\ref{l:separated nets}, there are
   \(\epsilon\)-separated \(\epsilon\)-nets, \(A\) of \(B_M(x,R)\) and \(B\) of
   \(B_N(y,R)\), such that \(x\in A\) and \(y\in B\).
   
   For each \(u\in B_M(x,R)\), there is \(v\in N\) such that
   \(d(u,v)<r_0+s_0\). Then there is \(v'\in B\) so that
   \(d_N(v,v')\le\epsilon\). So
     \[
       d(u,v')\le d(u,v)+d_N(v,v')<r_0+s_0+\epsilon,
     \]
     which implies \(d(u,B)<r_0+s_0+\epsilon\). Similarly, for all
       \(v\in B_N(y,R)\), \(d(v,A)<r_0+s_0+\epsilon\).
   
   Let \(\Sigma\) denote the set of pairs \((u,v)\in A\times B\) such that
   \(d(u,v)<r_0+s_0+\epsilon\) and \(\min\{d_M(x,u),d_N(y,v)\}<R\); in
   particular, \((x,y)\in\Sigma\). The set \(\Sigma\) is finite because
   \(A\) and \(B\) are separated and \(d\) is proper. For each
   \((u,v)\in\Sigma\), let \(I_{u,v}\) be the interval
   \([0,d(u,v)]\subseteq \R\) of
   length \(d(u,v)\), and let \(d_{u,v}\) be its standard metric. Let
   \(h:\bigsqcup_{(u,v)\in\Sigma}\partial I_{u,v}\to M\sqcup N\) be a
   map that restricts to a bijection \(h:\partial I_{u,v}\to\{u,v\}\)
   for all \((u,v)\in\Sigma\). Then let
    \[
      \widehat{P}=(M\sqcup N)\cup_h\bigsqcup_{(u,v)\in\Sigma}I_{u,v}.
    \]
  The spaces \(M\), \(N\) and each \(I_{u,v}\) may be viewed as subspaces of
  \(\widehat{P}\); in particular, \(\partial I_{u,v}\equiv\{u,v\}\) in
  \(\widehat{P}\).  Let \(\widehat{P}\) be endowed with the metric \(\hat
  d\) whose restriction to \(M\sqcup N\) is \(d\), whose restriction to
  each \(I_{u,v}\) is \(d_{u,v}\), and such that, for all \((u,v),(u',v')\in\Sigma\), \(w\in I_{u,v}\) and \(w'\in I_{u',v'}\),
    \begin{multline*}
      \hat d(w,w')=\min\bigl\{d_{u,v}(w,u)+d_M(u,u')+d_{u',v'}(u',w'),\\
      d_{u,v}(w,v)+d_N(v,v')+d_{u',v'}(v',w')\bigr\}.
    \end{multline*}
  
  Let \(P\subseteq\widehat{P}\) be the finite subset consisting of the
  points \(w\in I_{u,v}\) with \((u,v)\in\Sigma\) and
    \[
    d_{u,v}(w,u)=\frac{r_0+\epsilon}{r_0+s_0+2\epsilon}\,d(u,v).
    \]
    Let \(z\) be the unique point in \(P\cap I_{x,y}\), and consider
    the restriction of \(\hat d\) to \(P\).
  
    If \((u,v)\in\Sigma\) and \(w\) is the unique point in \(P\cap
    I_{u,v}\), then
    \[
      \hat d(u,w)\le d_{u,v}(u,w)
      <\frac{r_0+\epsilon}{r_0+s_0+2\epsilon}\,d(u,v)<r_0+\epsilon.
    \]
    So \(\hat d(x,z)<r_0+\epsilon<r\), and, for all \(u\in A\) and
    \(w\in P\), \(\hat d(u,P)<r_0+\epsilon\) and \(\hat
  d(w,M)<r_0+\epsilon\). Since \(A\) is an \(\epsilon\)-net in
  \(B_M(x,R)\), it also follows that, for all \(u\in B_M(x,R)\),
  \(\hat d(u,P)<r_0+2\epsilon\). Similarly, \(\hat d(y,z)<s\),
  and, for all \(v\in B_N(y,R)\) and \(w\in P\), \(\hat
  d(v,P)<s_0+2\epsilon\) and \(\hat d(w,N)<s_0+\epsilon\). Thus
    \[
      H_{\hat d,R}(M,x;P,z)\le r_0+2\epsilon<r,\quad H_{\hat
        d,R}(N,y;P,z)\le s_0+2\epsilon<s,
    \]
    and so \([P,z]\in U_{R,r}(M,x)\cap U_{R,s}(N,y)\) by
    Lemma~\ref{l:U_R,r}. Therefore \([N,z]\in U_{R,s}\circ
    U_{R,r}(M,x)\).
\end{proof}

Hypothesis~\ref{h:type I}-(ii) results from the next corollary.

\begin{cor}\label{c:E_r is cont}
  For \(R,r,s>0\), \(U_{T,r}\circ E_s\subseteq E_s\circ U_{R,r}\), where
  	\[
		T=R+2r+s+2\max\{r,s\}.
	\]
\end{cor}

\begin{proof}
  Let \(S=R+2r+s\). By Lemmas~\ref{l:U_S,r circ U_S,s subset
    U_R,r+s},~\ref{l:U_S,s circ U_R,r subset E_s circ U_R,r}
  and~\ref{l:U_R,r+s subset U_R,s circ U_R,r},
    \[
      U_{T,r}\circ E_s\subseteq U_{T,r}\circ U_{T,s}\subseteq U_{S,r+s}\subseteq U_{S,s}\circ U_{R,r}\subseteq E_s\circ U_{R,r}.\qed
    \]
\renewcommand{\qed}{}
\end{proof}

Hypothesis~\ref{h:type I}-(iii) is the statement of the next lemma.

\begin{lemma}\label{l:E_r cap(E_s(M,x) times E_t(M,x))}
  Let \([M,x]\in \MM_*\), \(s>0\) and \([N,y]\in E_s(M,x)\). If
  \(r>0\) and \(V\) is a neighborhood of \([N,y]\) in \(\MM_*\), then
  there is  a
  neighborhood \(W\) of \([N,y]\) in \(\MM_*\) such that
        \[
          E_r(W)\cap E_r(E_s(M,x))\subseteq E_r(V\cap E_s(M,x)).
        \]
\end{lemma}

\begin{proof}
  By Lemma~\ref{l:d'}, there are \(S>0\) and an open neighborhood
  \(W\) of \([N,y]\) in \(\MM_*\) such that, for all \([N',y']\in\MM_*\) and
  \([N'',y'']\in W\), if \((B_{N'}(y',S),y')\) is isometric to
  \((B_{N''}(y'',S),y'')\), then \([N',y']\in V\). Since \([N,y]\in
  U_{T,s}(M,x)\) for \(T=S+s+r\), we  also assume that \(W\subseteq
  U_{T,s}(M,x)\).
  
  For each \([P,z]\in E_r(W)\cap E_r(E_s(M,x))\), there are
  \([N_1,y_1]\in W\) and \([N_2,y_2]\in E_s(M,x)\) such that
  \([P,z]\in E_r(N_1,y_1)\cap E_r(N_2,y_2)\), and admissible, proper
  (by Lemma~\ref{l:d'}) metrics, \(d_1\) on \(M\sqcup N_1\) and \(\bar
  d_1\) on \(N_1\sqcup P\), so that \(d_1(x,y_1)<s\),
  \(H_{d_1,T}(M,x;N_1,y_1)<s\), \(\bar d_1(y_1,z)<r\) and \(H_{\bar
    d_1,T}(N_1,y_1;P,z)<r\). Let \((T_n)\) be a sequence in \(\R\)
  with \(T_n\uparrow\infty\) and \(T_0>T\); set also \(T_{-1}=T\). For
  each \(n\in\N\), let \((N_{2,n},y_{2,n})\) be an isometric copy of
  \((N_2,y_2)\). Then there are admissible, proper (by
  Lemma~\ref{l:d'}) metrics, \(d_{2,n}\) on \(M\sqcup N_{2,n}\) and
  \(\bar d_{2,n}\) on \(N_{2,n}\sqcup P\), such that
  \(d_{2,n}(x,y_{2,n})<s\),
  \(H_{d_{2,n},T_n}(M,x;N_{2,n},y_{2,n})<s\), \(\bar
  d_{2,n}(y_{2,n},z)<r\) and \(H_{\bar
    d_{2,n},T_n}(N_{2,n},y_{2,n};P,z)<r\).
  
  Let \(d\) denote the metric on \(M\sqcup
  N_1\sqcup(\bigsqcup_{n=0}^\infty N_{2,n})\sqcup P\) which extends
  \(d_1\), \(\bar d_1\), \(d_{2,n}\) and \(\bar d_{2,n}\) for all \(n\in\N\),
  and such that, for \(u\in M\) and \(w\in P\),
    \begin{multline*}
      d(u,w)=\inf\{\,d_1(u,v_1)+\bar d_1(v_1,w),\\
      d_{2,n}(u,v_{2,n})+\bar d_{2,n}(v_{2,n},w)\mid v_1\in N_1,\ v_{2,n}\in N_{2,n},\ n\in\N\,\bigr\},
    \end{multline*}
  for \(v_1\in N_1\) and \(v_{2,n}\in N_{2,n}\),
    \begin{multline*}
      d(v_1,v_{2,n})=\inf\bigl\{\,d_1(v_1,u)+d_{2,n}(u,v_{2,n}),\\ \bar
      d_1(v_1,w)+\bar d_{2,n}(w,v_{2,n})\mid u\in M,\ w\in P\,\bigr\},
    \end{multline*}
  and, for \(v_{2,m}\in N_{2,m}\) and \(v_{2,n}\in N_{2,n}\) with \(m\neq n\),
    \begin{multline*}
      d(v_{2,m},v_{2,n})=\inf\bigl\{\,d_{2,m}(v_{2,m},u)+d_{2,n}(u,v_{2,n}),\\ \bar
      d_{2,m}(v_{2,m},w)+\bar d_{2,n}(w,v_{2,n})\mid u\in M,\ w\in P\,\bigr\}.
    \end{multline*}
  Since the metrics \(d_1\), \(\bar d_1\), and
  \(d_{2,n}\) and \(\bar d_{2,n}\) for all \(n\in\N\), are proper,
  the metric \(d\) is proper as well. The set
    \[
      N'=\overline{B_{N_1}(y_1,T)}\sqcup\left(\bigsqcup_{n=0}^\infty\left(\overline{B_{N_{2,n}}(y_{2,n},T_n)}\setminus
      B_{N_{2,n}}(y_{2,n},T_{n-1})\right)\right)
    \]
  is closed in \(M\sqcup N_1\sqcup(\bigsqcup_{n=0}^\infty
  N_{2,n})\sqcup P\), and therefore it becomes a proper metric space
  with the restriction of \(d\).
  
  Then \(H_d(M,x;N',y_1)<s\) and \(H_d(N',y_1;P,z)<r\), as in
  Claim~\ref{cl:d_GH(N,y;P',z)<s}, and so \(d_{GH}(M,x;N',y_1)<s\) and
  \(d_{GH}(N',y_1;P,z)<r\) by~\eqref{d_GH(M,x;N,y) le ...}, which in
  turn implies \([N',y_1]\in E_s(M,x)\cap E_r(P,z)\) by
  Lemma~\ref{l:B_{GH}(M,x;r)}. On the other hand, like in
  Claim~\ref{cl:B_M'(y,R'+2epsilon)=B_N(y,R'+2epsilon)}, it follows
  that \(B_{N'}(y_1,S)=B_{N_1}(y_1,S)\) and so \([N',y_1]\in V\)
  because \([N_1,y_1]\in W\). Therefore \([P,z]\in E_r(V\cap
  E_s(M,x))\).
\end{proof}

Hypothesis~\ref{h:turbulent}-(i) is plainly true: the relation
\(E_{GH}\) has more than one equivalence class because the GH
distance between a bounded metric space and an unbounded one is always
infinite.

Hypothesis~\ref{h:turbulent}-(ii) is a consequence of the next lemma.

\begin{lemma}\label{l:E_GH is minimal}
  If \([M,x], [N,y]\in\MM_*\) and \(R, r> 0\), then there is \(s>0\)
  such that \(U_{R,r}(M,x)\cap E_s(N,y)\neq\emptyset\).
\end{lemma}

\begin{proof}
  Let \(A\) and \(B\) denote the balls of radius \(R+2r\) in \(M\) and \(N\)
  with centers \(x\) and \(y\), respectively. Let 
  \(s_0>d_{GH}(A,x;B,y)\) and let \(d\) be an admissible metric on \(A\sqcup
  B\) such that \(d(x,y)<s_0\) and \(H_d(A,B)<s_0\). Let \(d'\) be the
  admissible metric on \(M\sqcup N\) satisfying, for all \(u\in M\) and \(v\in N\),
    \[
      d'(u,v)=\inf\{\,d_M(u,u')+d(u',v')+d_N(v',v)\mid u'\in A,\ v'\in B\,\}.
    \]
  By the proof of Lemma~\ref{l:d'}, the metric  \(d'\) is proper, and its
  restriction to \(A\sqcup B\) equals \(d\); in particular, \(d'(x,y)<s_0\).
  
  Let \(A'\) and \(B'\) denote the balls of radius \(R+2r+s_0\) in \(M\)
  and \(N\) with centers \(x\) and \(y\), respectively. The set
  \(N'=\overline{A'}\sqcup(N\setminus B')\) is closed in \(M\sqcup N\),
  and therefore it becomes a proper metric space with the restriction
  of \(d'\). Take any
    \[
      s>\max\{s_0,R+2r+d'(x,N\setminus B')\}.
    \]
    If \(N\setminus B'\neq\emptyset\), then, for all \(v\in N\setminus
    B'\) and \(u\in\overline{A'}\),
    \[
      d'(u,v)\le d_M(u,x)+d'(x,v)<R+2r+d'(x,v),
    \]
  and so
    \[
      H_{d'}(\overline{A'},N\setminus B')\le R+2r+d'(x,N\setminus B')<s.
    \]
  It follows that \(H_{d'}(N,N')<s\), and so \(d_{GH}(N,y;N',x)<s\)
  by~\eqref{d_GH(M,x;N,y) le ...}. Therefore \([N',x]\in E_s(N,y)\) by
  Lemma~\ref{l:B_{GH}(M,x;r)}. Also, like in
  Claim~\ref{cl:B_M'(y,R'+2epsilon)=B_N(y,R'+2epsilon)}, 
  \(B_{N'}(x,R+2r)=A\), and therefore \([N',x]\in U_{R,r}(M,x)\) by
  Lemma~\ref{l:d'}.
\end{proof}

Hypothesis~\ref{h:turbulent}-(iii) is verified as follows. Let
\(R,r>0\) and \([M,x]\in\MM_*\). Let \(S>R\) and \(s>0\) be such that
\(s<r\) and \(R+2\max\{s,r-s\}<S\), and let \(\DD\) denote the set of points
\([N,y]\in\MM_*\) such that there is some admissible metric, \(d\), on
\(M\sqcup N\) so that \(d(x,y)<s\), \(H_{d,S}(M,x;N,y)<s\) and
\(H_d(M,N)<\infty\). 
  
\begin{lemma}\label{l:DD is dense}
  \(\DD\) is a dense subset of \(U_{S,s}(M,x)\cap E_{GH}(M,x)\).
\end{lemma}

\begin{proof}
  By its definition, the set \(\DD\subseteq U_{S,s}(M,x)\cap
  E_{GH}(M,x)\).  It must to be shown that, for every \(T,t,t'>0\) and
  \([N,y]\in U_{S,s}(M,x)\cap B_{GH}(M,x;t')\), the intersection
  \(U_{T,t}(N,y)\cap \DD\ne \emptyset\). Let \((N_1,y_1)\) and
  \((N_2,y_2)\) be two isometric copies of \((N,y)\). There are
  admissible metrics, \(d_1\) on \(M\sqcup N_1\) and \(d_2\) on
  \(M\sqcup N_2\), such that \(d_1(x,y_1)<s\),
  \(H_{d_1,R}(M,x;N_1,y_1)<s\), \(d_2(x,y_2)<t'\) and
  \(H_{d_2}(M,N_2)<t'\). Let \(\hat d\) denote the metric on \(M\sqcup
  N_1\sqcup N_2\) whose restrictions to \(M\sqcup N_1\) and \(M\sqcup
  N_2\) are \(d_1\) and \(d_2\), respectively, and such that, for all \(v_1\in N_1\) and \(v_2\in N_2\),
    \[
      \hat d(v_1,v_2)=\inf\{\,d_1(v_1,u)+d_2(u,v_2)\mid u\in M\,\}.
    \]
  Since \(d_2\) is proper by
  Lemma~\ref{l:d is proper}, and \(d_1\) can be assumed to be proper by
  Lemma~\ref{l:d'}, the metric \(\hat d\) is proper as well. Let
    \[
      T'=\max\{S,T\}+2\max\{s,t\}+t'+s,
    \]
  let \(A=B_M(x,T'+2t')\), \(B_1=B_{N_1}(y_1,T')\) and
  \(B_2=B_{N_2}(y_2,T')\), and set
  \(N'=\overline{B_1}\sqcup(N_2\setminus B_2)\). Since \(N'\) is closed in
  \(M\sqcup N_1\sqcup N_2\), it becomes a proper metric space with the
  restriction of \(\hat d\). We have \(\hat d(x,y_1)=d_1(x,y_1)<s\). With
  arguments used in Claims~\ref{cl:d_GH(N,y;P',z)<s}
  and~\ref{cl:B_M'(y,R'+2epsilon)=B_N(y,R'+2epsilon)}, we obtain \(H_{\hat
    d,R}(M,x;N',y_1)<s\) and
    \[
      H_{\hat d}(M,N')<\max\bigl\{H_{d_1}(A,B_1),t'\bigr\}<\infty.
    \]
   It follows that \([N',y_1]\) satisfies the condition to be in \(\DD\)
   with the restriction of \(\hat d\) to the subset \(M\sqcup N'\) of
   \(M\sqcup N_1\sqcup N_2\). Also, since \(\hat
   d(y_1,y_2)\le t'+s\), the proof of
   Claim~\ref{cl:B_M'(y,R'+2epsilon)=B_N(y,R'+2epsilon)} leads to
    \[
      B_{N'}(y_1,T+2t)=B_{N_1}(y_1,T+2t)\equiv B_N(y,T+2t),
    \]
  and therefore \([N',y_1]\in U_{T,t}(N,y)\) by Lemma~\ref{l:d'}.
\end{proof}

Let \(\EE\) be the set of those \([M,x]\in\DD\) such that \(M\) is
separated (in itself). It easily follows from Lemma~\ref{l:separated
  nets} that \(\EE\) is \(d_{GH}\)-dense in \(\DD\). Take any
\(\epsilon>0\) such that \(s+2\epsilon<r\) and
\(R+2\max\{s+\epsilon,r-s-\epsilon\}<S\).  Let \(A\) be a separated
\(\epsilon\)-net of \(M\) that contains \(x\), whose existence is
guaranteed by Lemma~\ref{l:separated nets}, and consider the
restriction of \(d_M\) to \(A\). Observe that \([A,x]\in
E_{r-s-\epsilon}(M,x)\) because \(r-s-\epsilon>\epsilon\). Then the
proof of Hypothesis~\ref{h:turbulent}-(iii) is completed by the
following lemma.

\begin{lemma}\label{l:EE}
  Any point of \(\EE\) can be joined to \([A,x]\) by a
  \(d_{GH}\)-continuous path in \(U_{R,r}(M,x)\).
\end{lemma}

\begin{proof}
  For any \([N,y]\in\EE\), there is some admissible metric, \(d\), on
  \(M\sqcup N\) such that \(d(x,y)<s\), \(s_0:=H_{d,S}(M,x;N,y)<s\) and
  \(s_1:=H_d(M,N)<\infty\). Moreover \(d\) is proper by Lemma~\ref{l:d is
    proper}. Observe that \(H_{d,S}(A,x;N,y)<s_0+\epsilon\) and
  \(H_d(A,N)\le s_1+\epsilon\).
  
  Let \(\Sigma\) be the set of pairs \((u,v)\in A\times N\) such that
  \(d(u,v)\le s_1+\epsilon\) and, if \(u\in B_A(x,S)\) or \(v\in B_N(y,S)\),
  then \(d(u,v)\le s_0+\epsilon\); in particular, \((x,y)\in\Sigma\). Like
  in the proof of Lemma~\ref{l:U_R,r+s subset U_R,s circ U_R,r},
  define \(I_{u,v}\) and \(d_{u,v}\) for each \((u,v)\in\Sigma\), as well as
  \(h:\bigsqcup_{(u,v)\in\Sigma}\to A\sqcup N\),
    \[
      \widehat{P}=(A\sqcup N)\cup_h\bigsqcup_{(u,v)\in\Sigma}I_{u,v},
    \]
  and the metric \(\hat d\) on \(\widehat{P}\). Since \(d\) is proper and
  \(A\) and \(N\) are separated,  the \(d\)-balls in \(A\sqcup
  N\) are finite. Therefore, any ball in \(\widehat{P}\) is contained
  in a finite union of segments \(I_{u,v}\), and so
  \(\widehat{P}\) is proper.
  
  For each \(t\in I=[0,1]\), let \(P_t\subseteq\widehat{P}\) be the subset
  consisting of the points \(w\in I_{u,v}\) with
  \(d_{u,v}(w,u)=t\,d(u,v)\) for \((u,v)\in\Sigma\), and let \(z_t\) denote
  the unique point of \(P_t\cap I_{x,y}\).  Each \(P_t\) is a discrete
  subspace of \(\widehat{P}\), and it therefore becomes a proper metric
  space with the restriction of \(\hat d\). Moreover \((P_0,z_0)=(A,x)\)
  and \((P_1,z_1)=(N,y)\). For all \(t,t'\in I\), \((u,v)\in\Sigma\), \(w\in
  P_t\cap I_{u,v}\) and \(w'\in P_{t'}\cap I_{u,v}\),
    \begin{multline}\label{hat d(u,v) le ...}
      \hat d(w,w')=d_{u,v}(w,w')=d(u,v)\,|t-t'|\\
      \le
        \begin{cases}
          (s_1+\epsilon)\,|t-t'| & \text{for arbitrary \((u,v)\in\Sigma\)}\\
          (s_0+\epsilon)\,|t-t'| & \text{if \(u\in B_A(x,S)\) or \(v\in B_N(y,S)\)}.
        \end{cases}
    \end{multline}
  Thus \(\hat d(z_t,z_{t'})\le(s_0+\epsilon)\,|t-t'|\) and \(H_{\hat
    d}(P_t,P_{t'})\le(s_1+\epsilon)\,|t-t'|\). By~\eqref{d_GH(M,x;N,y)
    le ...}, it follows that \([P_t,z_t]\in E_{GH}(M,x)\) for all \(t\in
  I\), and the mapping \(t\mapsto[P_t,z_t]\) is \(d_{GH}\)-continuous.
  
  From~\eqref{hat d(u,v) le ...}, it also follows that \(\hat
  d(u,P_t)\le(s_0+\epsilon)t\) for all \(u\in B_A(x,S)\) and \(t\in
  I\). Moreover the ball \(B_{P_t}(z_t,S)\) is contained in the union of
  the segments \(I_{u,v}\) for \((u,v)\in\Sigma\) with \(u\in B_A(x,S)\) or
  \(v\in B_N(y,S)\). So \(\hat d(w,P_t)\le(s_0+\epsilon)t\) for all \(w\in
  B_P(z_t,S)\) by~\eqref{hat d(u,v) le ...}. It follows that
    \[
      H_{\hat d,S}(A,x;P_t,z_t)\le(s_0+\epsilon)t<s+\epsilon,
    \]
  obtaining
    \[
      [P_t,z_t]\in U_{S,s+\epsilon}(A,x)\subseteq U_{S,s+\epsilon}\circ E_{r-s-\epsilon}(M,x)\subseteq U_{R,r}(M,x)
      \]
  by Lemmas~\ref{l:U_R,r} and~\ref{l:U_S,r circ U_S,s subset U_R,r+s}.
\end{proof}
    
Hypotheses~\ref{h:U_R,r}-\ref{h:turbulent} have just been proved, and
that suffices to confirm Theorem~\ref{t: d_infty, d_GH and d_QI} for
\((d_{GH},E_{GH})\).

\begin{rem}
  As in Remark~\ref{r:B_infty(f,r) is not Polish}, it can be proved
  that, for all \(r>0\), \(\emptyset\) is residual in \(B_{GH}(M,x;r)\) if \(M\)
  is unbounded. In this case, for sequences \(0<r_n\uparrow r\) and
  \(0<R_n\uparrow\infty\), consider the sets \(U_n\) consisting of the
  points \([N,y]\in B_{GH}(M,x;r)\) such that
  \[
  H_d\left(M\setminus\overline{B_M(x,R_n)},N\setminus\overline{B_N(y,R_n)}\right)>r_n
  \]
  for every admissible metric, \(d\), on \(M\sqcup N\).
\end{rem}

\section{The QI metric relation}\label{s:QI}


Since \(d_{QI}\le d_{GH}\) and Theorem~\ref{t: d_infty, d_GH and d_QI}
is already proved for \(d_{GH}\), Remarks~\ref{r:generalized
  pseudo-metric relations} and~\ref{r:local equivalence equivalence
  classes} imply that the proof of Theorem~\ref{t: d_infty, d_GH and
  d_QI} for \(d_{QI}\) only requires the next proposition.

\begin{prop}\label{p:the fibers of E_QI are meager}
  The fibers of \(E_{QI}\) are meager in \(\MM_*\).
\end{prop}

The proof of Proposition~\ref{p:the fibers of E_QI are meager}
requires an analysis of \(d_{QI}\), which in turn requires an analysis
of \(d_{GH}\) and \(d_{\text{\rm Lip}}\). 

A map between metric spaces, \(\phi:M\to N\), is called
\emph{bi-Lipschitz} if there is some \(\lambda\ge1\) such that
  \[
    \frac{1}{\lambda}\,d_M(u,v)\le d_N(\phi(u),\phi(v))\le \lambda\,d_M(u,v)
  \]
for all \(u,v\in M\). The term \emph{\(\lambda\)-bi-Lipschitz} may be also used in this case. 

A (\emph{coarse}) \emph{quasi-isometry} of \(M\) to
\(N\) is a bi-Lipschitz bijection \(\phi:A\to B\) for nets \(A\subseteq
M\) and \(B\subseteq N\). The existence of a quasi-isometry of \(M\) to
\(N\) is equivalent to the existence of a finite sequence of metric
spaces \(M=M_0,\dots,M_{2k}=N\) such that
\(d_{GH}(M_{2i-2},M_{2i-1})<\infty\) and such that there is a bi-Lipschitz
bijection \(M_{2i-1}\to M_{2i}\) for each \(i\in\{1,\dots,k\}\). A
\textit{pointed} (\textit{coarse}) \textit{quasi-isometry} is defined in the
same way, by using a pointed bi-Lipschitz bijection between nets that
contain the distinguished points. The existence of a pointed
quasi-isometry has an analogous characterization involving pointed
Gromov-Hausdorff distances and pointed bi-Lipschitz bijections.

As noted in Section~\ref{s:Gromov sp}, \(d_{\text{\rm Lip}}\) is the
metric equivalence relation over \(\MM_*\) defined by setting
\(d_{\text{\rm Lip}}(M,x;N,y)\) equal to the infimum of the set of \(r\ge0\)
for which there is a pointed \(e^r\)-bi-Lipschitz bijection
\(\phi:(M,x)\to(N,y)\). 

The distance \(d_{QI}(M,x;N,y)\) equals the infimum of all sums
  \[
    \sum_{i=1}^kd_{GH}(M_{2i-2},x_{2i-2};M_{2i-1},x_{2i-1})+d_{\text{\rm Lip}}(M_{2i-1},x_{2i-1};M_{2i},x_{2i})
  \]
for finite sequences \([M,x]=[M_0,x_0],\dots,[M_{2k},x_{2k}]=[N,y]\) in
\(\MM_*\). For
\([M,x]\in\MM_*\) and \(r>0\), the notation \(B_{\text{\rm
    Lip}}(M,x;r)=B_{d_{\text{\rm Lip}}}([M,x],r)\) and
\(B_{QI}(M,x;r)=B_{d_{QI}}([M,x],r)\) will be used.

\begin{lemma}\label{l:[N,y] in U_R,r(M,x)<r}
  For \(r>0\) and \(R\ge q>p>2r\), if \([N,y]\in U_{R,r}(M,x)\) and
  \(B_M(x,q)\setminus\overline{B_M(x,p)}\ne\emptyset\), then
  \(B_N(y,q+2r)\setminus\overline{B_N(y,p-2r)}\ne\emptyset\).
\end{lemma}

\begin{proof}
  By hypothesis, there is an admissible metric \(d\) on \(M\sqcup N\) such
  that \(d(x,y)<r\) and \(H_{d,R}(M,x;N,y)<r\), and there is \(u\in M\)
  such that \(p<d(x,u)<q\). Since \(H_{d,R}(M,x;N,y)<r\), there is 
  \(v\in N\) such that \(d(u,v)<r\). Then
    \[
      d_N(y,v)\le d(x,u)+d(y,x)+d(u,v)<q+2r,
    \]
  and, similarly, \(d_N(y,v)>p-2r\).
\end{proof}

\begin{cor}\label{c:d_GH(M,x;N,y)<r}
  If \(d_{GH}(M,x;N,y)<r\) and \(q>p>2r\) are such that
  \(B_M(x,q)\setminus\overline{B_M(x,p)}\ne\emptyset\), then
  \(B_N(y,q+2r)\setminus\overline{B_N(y,p-2r)}\ne\emptyset\).
\end{cor}

\begin{lemma}\label{l:d_Lip(M,x;N,y)<r}
  If \(d_{\text{\rm Lip}}(M,x;N,y)<r\) and \(p>q>0\) are such that
  \(B_M(x,q)\setminus\overline{B_M(x,p)}\ne\emptyset\), then
  \(B_N(y,e^rq)\setminus\overline{B_N(y,e^{-r}p)}\ne\emptyset\).
\end{lemma}

\begin{proof}
  By hypothesis, there is a pointed \(e^r\)-bi-Lipschitz bijection
  \(\phi:(M,x)\to(N,y)\), and there is
  \(u\in M\) such that \(p<d(x,u)<q\). Then
    \[
      d_N(y,\phi(u))\le e^rd_M(x,u)<e^rq,
    \]
  and, similarly, \(d_N(y,\phi(u))>e^{-r}p\), showing the result.
\end{proof}

\begin{proof}[Proof of Proposition~\ref{p:the fibers of E_QI are meager}]
  The pointed compact metric spaces form an equivalence class of \(E_{GH}\)
  which is meager in \(\MM_*\) by Theorem~\ref{t: d_infty, d_GH and d_QI}-(i) for \((d_{GH},E_{GH})\). Moreover any
  metric space bi-Lipschitz equivalent to a bounded one is also
  bounded. So the pointed compact metric spaces also form a class of
  \(E_{QI}\). Thus, to prove Proposition~\ref{p:the fibers of E_QI are
    meager}, it is enough to consider the fiber \(E_{QI}(M,y)\) for any
  unbounded proper metric space \(M\). Hence there are sequences
  \(p_n,q_n\uparrow\infty\) such that \(q_n>p_n>0\) and
  \(B_M(x,q_n)\setminus\overline{B_M(x,p_n)}\ne\emptyset\).
    
 \begin{claim}\label{cl:[N,y] in overline B_QI(M,x;r)}
   Let \(r,s>0\) and \(n\in\N\) so that \(p_n>2r\) and
   \(2s<e^{-r}(q_n-2r)\). If \([N,y]\in\overline{B_{QI}(M,x;r)}\),
   then
   \begin{equation}\label{[N,y] in overline B_QI(M,x;r)}
     B_N(y,e^r(q_n+2r)+2s)\setminus\overline{B_N(y,e^{-r}(p_n-2r)-2s)}\ne\emptyset.
   \end{equation}
 \end{claim}
    
    For the proof, let \(S>e^r(q_n+2r)\). Since
    \([N,y]\in\overline{B_{QI}(M,x;r)}\), there is a finite sequence,
    \([M,x]=[M_0,x_0],\dots,[M_{2k},x_{2k}]\) in \(\MM_*\) such that
    \([M_{2k},x_{2k}]\in U_{S,s}(N,y)\) and
    \[
      \sum_{i=1}^kd_{GH}(M_{2i-2},x_{2i-2};M_{2i-1},x_{2i-1})+d_{\text{\rm Lip}}(M_{2i-1},x_{2i-1};M_{2i},x_{2i})<r.
    \]
  Let \(r_1,\dots,r_{2k}>0\) be such that \(\sum_{j=1}^{2k}r_j<r\) and, for \(j\in\{1,\dots,2k\}\),
    \[
      r_j>
        \begin{cases}
          d_{GH}(M_{j-1},x_{j-1};M_j,x_j) & \text{if \(j\) is odd}\\
          d_{\text{\rm Lip}}(M_{j-1},x_{j-1};M_j,x_j) & \text{if \(j\) is even}.
        \end{cases}
    \]
  Let \(\bar r_j=\sum_{a=1}^jr_a\). Arguing by
  induction on \(j\), using Corollary~\ref{c:d_GH(M,x;N,y)<r} and
  Lemma~\ref{l:d_Lip(M,x;N,y)<r}, it follows that
    \[
    B_{M_j}(x_j,e^{\bar r_j}(q_n+2\bar
    r_j))\setminus\overline{B_{M_{2k}}(x_{2k},e^{-\bar r_j}(q_n-2\bar
      r_j))}\ne\emptyset
    \]
  for all \(j\). So
    \[
    B_{M_{2k}}(x_{2k},e^r(q_n+2r))\setminus\overline{B_{M_{2k}}(x_{2k},e^{-r}(q_n-2r))}\ne\emptyset.
    \]
  Then~\eqref{[N,y] in overline B_QI(M,x;r)} follows by
  Lemma~\ref{l:[N,y] in U_R,r(M,x)<r}, completing the proof of
  Claim~\ref{cl:[N,y] in overline B_QI(M,x;r)}.

    \begin{claim}\label{cl:B_QI(M,x;r) is nowhere dense}
      For each \(r>0\), \(B_{QI}(M,x;r)\) is nowhere dense in \(\MM_*\).
    \end{claim}
    
  Let \([N,y]\in\overline{B_{QI}(M,x;r)}\). Given \(S,s>0\), there is some
  \(n\in\N\) such that \(p_n>2r\) and
  \(S<e^{-r}(q_n-2r)-2s\). Thus~\eqref{[N,y] in overline B_QI(M,x;r)} is
  satisfied with these \([N,y]\), \(r\), \(s\) and \(n\). Let
    \[
      N'=N\setminus\left(B_N(y,e^r(q_n+2r)+2s)\setminus\overline{B_N(y,e^{-r}(q_n-2r)-2s)}\right).
    \]
  With the restriction of \(d_N\), \(N'\) is a proper metric
  space with \(B_{N'}(y,S)=B_N(y,S)\), obtaining \([N',y]\in
  U_{S,s}\). But \([N',y]\not\in\overline{B_{QI}(M,x;r)}\) by
  Claim~\ref{cl:[N,y] in overline B_QI(M,x;r)} because
    \[
      B_{N'}(y,e^r(q_n+2r)+2s)\setminus\overline{B_{N'}(y,e^{-r}(p_n-2r)-2s)}=\emptyset.
    \]
  So \(U_{S,s}(N,y)\not\subseteq\overline{B_{QI}(M,x;r)}\). Then
  Claim~\ref{cl:B_QI(M,x;r) is nowhere dense} follows since \(s\) can be
  chosen arbitrarily small, and \(S\) arbitrarily large by
  choosing \(n\) arbitrarily large.

  Since \(E_{QI}(M,x)=\bigcup_{r=1}^\infty B_{QI}(M,x;r)\),
  Claim~\ref{cl:B_QI(M,x;r) is nowhere dense} concludes the proof of
  Proposition~\ref{p:the fibers of E_QI are meager}.
\end{proof}

\bibliographystyle{amsplain}



\providecommand{\bysame}{\leavevmode\hbox to3em{\hrulefill}\thinspace}
\providecommand{\MR}{\relax\ifhmode\unskip\space\fi MR }
\providecommand{\MRhref}[2]{%
  \href{http://www.ams.org/mathscinet-getitem?mr=#1}{#2}
}
\providecommand{\href}[2]{#2}

\end{document}